\documentclass[11pt,a4paper,leqno,twoside,notitlepage]{article}

\usepackage{amsmath}
\usepackage{amsfonts}
\usepackage{amssymb}
\usepackage{amsxtra}
\usepackage{amsthm}
\usepackage{enumerate}
\usepackage{ifthen}

\newboolean{details}
\newboolean{appendix}

\newenvironment{enumerateroman}
{\begin{enumerate}\renewcommand{\labelenumi}%
{\textit{(\roman{enumi})}}}%
{\end{enumerate}}

\numberwithin{equation}{section}

\theoremstyle{plain}
\newtheorem{thm}[equation]{Theorem}
\newtheorem{prop}[equation]{Proposition}
\newtheorem{lem}[equation]{Lemma}

\newtheorem{obs}[equation]{Observation}

\theoremstyle{definition}
\newtheorem{defn}[equation]{Definition}
\newtheorem{rem}[equation]{Remark}
\newtheorem*{ack}{Acknowledgements}

\newcommand{\wrt}{w.\,r.\,t.~}

\newcommand{\mC}{{\mathbb C}}

\newcommand{\mN}{{\mathbb N}}

\newcommand{\mR}{{\mathbb R}}
\newcommand{\mZ}{{\mathbb Z}}
\newcommand{\mcA}{\mathcal{A}}
\newcommand{\mcB}{\mathcal{B}}
\newcommand{\mcC}{\mathcal{C}}
\newcommand{\mcD}{\mathcal{D}}
\newcommand{\mcE}{\mathcal{E}}

\newcommand{\mcK}{\mathcal{K}}

\newcommand{\mcM}{\mathcal{M}}

\newcommand{\mcP}{\mathcal{P}}
\newcommand{\mcQ}{\mathcal{Q}}
\newcommand{\mcS}{\mathcal{S}}
\newcommand{\mcU}{\mathcal{U}}
\newcommand{\fraka}{\mathfrak{a}}
\newcommand{\frakb}{\mathfrak{b}}
\newcommand{\frakc}{\mathfrak{c}}

\newcommand{\frakg}{\mathfrak{g}}

\newcommand{\frakH}{\mathfrak{H}}

\newcommand{\frakm}{\mathfrak{m}}
\newcommand{\frakn}{\mathfrak{n}}
\newcommand{\frakp}{\mathfrak{p}}

\newcommand{\fraks}{\mathfrak{s}}
\newcommand{\frakt}{\mathfrak{t}}

\newcommand{\frakz}{\mathfrak{z}}

\newcommand{\LG}{L^1(G)}

\newcommand{\LM}{L^1(M)}

\newcommand{\CG}{C^{\ast}(G)}

\newcommand{\CM}{C^{\ast}(M)}

\newcommand{\Ccinfty}[1]{\mcC^{\infty}_0(#1)}
\newcommand{\Cinfty}[1]{\mcC^{\infty}(#1)}

\newcommand{\twoheadlongrightarrow}{\,\,\rightarrow%
\mspace{-24.0mu}\longrightarrow\,}

\renewcommand{\to}{\longrightarrow}

\newcommand{\onto}{\mbox{$\twoheadlongrightarrow$}}

\newcommand{\mdot}{\!\cdot\!}
\newcommand{\clos}{^{\textbf{---}}}

\DeclareMathOperator{\ind}{ind}
\DeclareMathOperator{\Ad}{Ad}
\DeclareMathOperator{\coAd}{Ad^{\ast}}
\DeclareMathOperator{\ad}{ad}

\DeclareMathOperator{\Id}{Id}

\DeclareMathOperator{\Der}{Der}
\DeclareMathOperator{\Hom}{Hom}

\DeclareMathOperator{\Prim}{Prim}

\DeclareMathOperator{\sgn}{sgn}

\DeclareMathOperator{\supp}{supp}

\DeclareMathOperator{\vol}{vol}

\setboolean{details}{false}
\setboolean{appendix}{false}

\begin{document}

\title{Invariant differential operators and central Fourier
multipliers on exponential Lie groups}
\author{O.\ Ungermann \\[0.5cm] \small{(Universit\"at Bielefeld)}}
\date{}
\maketitle
\hspace{1.5cm}\parbox{115mm}{\small\noindent\textbf{Abstract.}
Let $G$ be an exponential solvable Lie group. By definition $G$ is
$\ast$-regular if $\ker_{\LG}\pi$ is dense in $\ker_{\CG}\pi$ for
all unitary representations $\pi$ of~$G$. Boidol characterized the
$\ast$-regular exponential Lie groups by a purely algebraic condition.
In this article we will focus on non-$\ast$-regular groups. We say
that $G$ is primitive $\ast$-regular if the above density condition
is satisfied for all irreducible representations. Our goal is to
develop appropriate tools to verify this weaker property. To this
end we will introduce Duflo pairs $(W,p)$ and central Fourier
multipliers $\psi$ on the stabilizer $M=G_fN$ of representations
$\pi=\mcK(f)$ in general position. Using Littlewood-Paley theory
we will derive some results on multiplier operators $T_\psi$ which
might be of independent interest. The scope of our method of
separating triples $(W,p,\psi)$ will be sketched by studying two
significant examples in detail. It should be noticed that the
methods 'separating triples' and 'restriction to subquotients'
suffice to prove that all exponential solvable Lie groups of
dimension $\le 7$ are primitive $\ast$-regular.}\\[1cm]

\hspace{1.5cm}\parbox{115mm}{ \textit{MSC(2000):} 17B35, 22D20,
22E27, 42B15; 43A22, 46F05}

\section{Introduction}\label{sec:introduction}
Let $G$ be an exponential solvable Lie group with Lie algebra
$\frakg$. Fix a co\-abelian, nilpotent ideal $\frakn$ of $\frakg$.
Let $f\in\frakg^\ast$ be in general position such that
$\frakm=\frakg_f+\frakn$ is a proper, non-nilpotent ideal.
Set $\tilde{f}=f\,|\,\frakm$ and $f'=f\,|\,\frakn$. Further
we consider the orbit $\tilde{X}=\coAd(G)\tilde{f}$
and the $\coAd(G)$-invariant subset
$\tilde{\Omega}=\{\tilde{h}\in\frakm^\ast: \tilde{h}
\,|\,\frakn\text{ is in the closure of }\coAd(G)f'
\text{ in }\frakn^\ast\}$ of $\frakm^\ast$.\\\\
Let $M$ denote the connected subgroup of $G$
with Lie algebra $\frakm$. We work with the
$C^\ast$-completion $\CM$ of the Banach
$\ast$-algebra $\LM$. As usual we provide $\Prim\CM$
with the Jacobson topology, and $\widehat{M}$ with
the initial topology \wrt the natural map
$\widehat{M}\onto\Prim\CM$. This map is a bijection
by the main result of \cite{Pog3}. Furthermore it is
known that the Kirillov map yields a $G$-equivariant
bijection from the coadjoint orbit space
$\frakm^\ast/\coAd(M)$ onto the unitary dual
$\widehat{M}$ of $M$. Hence $\tilde{X}$ corresponds
to a $G$-orbit $X$ in~$\widehat{M}$, and
$\tilde{\Omega}$ to a $G$-invariant subset
$\Omega$ of~$\widehat{M}$.\\\\
Our aim is to prove $G$ to be primitive $\ast$-regular.
According to the strategy developed in Section~5
of~\cite{Ung1} we have to verify
\begin{equation}\label{cFm_equ:no_inclusion}
\bigcap\limits_{\pi\in X}\ker_{\LM}\pi\not\subset
\ker_{\LM}\rho
\end{equation}
for all critical $\rho\in\Omega\setminus\overline{X}$.
Compare also Assertion~5.3 of~\cite{Ung1}. Note that
$\Omega\setminus\overline{X}$ consists of all
representations $\rho=\mcK(\tilde{g})$ where
$\tilde{g}\in\tilde{\Omega}$ is critical for
$\tilde{X}=\coAd(G)\tilde{f}$ in the sense of
Definition~5.2 of~\cite{Ung1}.\\\\
Relation (\ref{cFm_equ:no_inclusion}) deserves a further
explanation. If the primitive ideal spaces $\Prim C^\ast(M)$
and $\Prim_\ast\LM$ are endowed with the Jacobson topology,
then the natural map $\Psi:\Prim C^\ast(M)\onto\Prim_\ast\LM$,
$\Psi(P)=P'=P\cap\mcA$, is a continuous bijection. Continuity
means that $\Psi(\overline{A})\subset\overline{\Psi(A)}$
for all subsets $A$ of $\Prim\CM$. Note that injectivity,
which is necessary for $M$ to be primitive $\ast$-regular,
follows from~\cite{Pog3}. We shall assume that $M$ is not
$\ast$-regular, or equivalently, that $\Psi$ is not a homeomorphism.
In this situation one may ask whether the reverse inclusion
$\overline{\Psi(X)}\subset\Psi(\overline{X})$ is satisfied at
least for \textit{certain} subsets $X$ of $\widehat{M}=\Prim\CM$.
The question is: Does $\rho\not\in\overline{X}$ imply
Relation~(\ref{cFm_equ:no_inclusion}) for all orbits
$X$ of an exponential Lie group $G$ as above?\\\\
Producing functions $c\in\LM$ such that $\pi(c)=0$ for
all $\pi\in X$ and $\rho(c)\neq 0$ is a challenging
problem. Our approach to a solution is best explained
in the context of the subsequent theorem.

\begin{thm}[Dauns, Dixmier, Hofmann;1967]
\label{cFm_thm:Dauns_Hofmann}
Let $\mcB$ be $C^\ast$-algebra, $\mcB^b$ its adjoint
algebra (multiplier algebra), and $\mcC^b(\Prim\mcB)$
the $C^\ast$-algebra of all complex-valued, bounded,
continuous functions on its primitive ideal space.
\begin{enumerateroman}
\item If $\mu\in\mcC^b(\Prim\mcB)$ and $a\in\mcB$, then there
exists a unique element $c=\mu\ast a$ in $\mcB$ such that
$c+P=\mu(P)\mdot(a+P)$ holds in $\mcB/P$ for all $P\in\Prim\mcB$.\\
\item Further $\Gamma(\mu)a=\mu\ast a$ defines an isomorphism
$\Gamma$ of $\mcC^b(\Prim\mcB)$ onto the center $Z(\mcB^b)$
of $\mcB^b$.
\end{enumerateroman}
\end{thm}

In particular $\mcC^b(\Prim\mcB)$ acts on $\mcB$ as an algebra
of multipliers. This theorem was first shown by Dauns and
Hofmann, see Chapter~3 of~\cite{Dauns1}. An alternative direct
proof is due to Dixmier, see Theorem~5 of~\cite{Dix5}. Compare
also pp.~223-226 of~\cite{RaeWil}. For a definition of the
adjoint algebra $\mcB^b$ we refer to~Section~3
of~\cite{Lep1}.\\\\
The spectrum $\widehat{\mcB}$ of $\mcB$ is the set of all
equivalence classes of irreducible $\ast$-representations of
$\mcB$ endowed with the initial topology \wrt the natural
map $\widehat{\mcB}\onto\Prim\mcB$. This map allows
us to consider $\mu$ as a function on~$\widehat{\mcB}$. Now
Theorem~\ref{cFm_thm:Dauns_Hofmann}.\textit{(i)} reads as
follows: There exists a unique element $c\in\mcB$ such that
$\pi(c)=\mu(\pi)\,\pi(a)$ for all $\pi\in\widehat{\mcB}$.
There is no peril in dealing with $\pi$ instead of its
equivalence class $[\pi]$ here. If we write
$\widehat{a}(\pi)=\pi(a)$, then the preceding equality
becomes $\widehat{c}(\pi)=\mu(\pi)\,\widehat{a}(\pi)$ so
that $\mu$ emerges as to be a multiplier (on the Gelfand
transform side). Theorem~\ref{cFm_thm:Dauns_Hofmann}
states that the multiplier problem given by $\mu$ and
$a$ has a (unique) solution $c\in\mcB$. If $\mcB=\CM$, then
one might wonder about the regularity of this solution:
What conditions on $\mu$ and $a$ do imply $c\in\LM\,$?\\\\
Finding suitable $\mu$ and proving this kind of
regularity is appropriate to tackle the problem raised
above. In the following remark we localize to a certain
subset $\Omega$ of $\widehat{M}$ losing the uniqueness
of~$c$. We do not assume $\mu$ to be bounded.

\begin{obs}\label{cFm_obs:regularity_for_multipliers}
Let $X\subset\Omega$ be subsets of $\widehat{M}$ and
$\rho\in\Omega\setminus\overline{X}$. If there exists
a dense subspace $\mcQ$ of $\LM$ and a complex-valued,
continuous function $\mu$ on $\Omega$ such that
\begin{enumerateroman}
\item $\mu(\pi)=0$ for all $\pi\in X$ and $\mu(\rho)\neq 0$,
\item for every $a\in\mcQ$ there exists a function $c\in\LM$
satisfying\\
 $\pi(c)=\mu(\pi)\,\pi(a)$ for all $\pi\in\Omega$,
\end{enumerateroman}
then it follows $\bigcap\limits_{\pi\in X}\ker_{\LM}\pi
\not\subset\ker_{\LM}\rho$.
\end{obs}
\begin{proof}
Since $\mcQ$ is dense in $\LM$, there exists some
$a\in\mcQ$ such that $\rho(a)\neq 0$. If we choose $\mu$
as in~\textit{(i)} and $c\in\LM$ as in~\textit{(ii)},
then $\pi(c)=0$ for all $\pi\in X$ and $\rho(c)\neq 0$.\qed
\end{proof}

This observation is the guideline for the results of
Section~\ref{sec:invariant_differential_operators}
and~\ref{sec:central_Fourier_multipliers}. In the course
of the proof of Theorem~\ref{cFm_thm:separating_triples}
we will see that if $(W,p,\psi)$ is a separating
triple for $X\subset\Omega\subset\widehat{M}$ and
$\rho\not\in\overline{X}$, then $\mu=p-\psi$
satisfies \textit{(i)} and \textit{(ii)} of
Observation~\ref{cFm_obs:regularity_for_multipliers}
so that Relation~(1) is guaranteed. The author has
verified the existence of separating triples in
a multitude of examples. A sample can be found
in Section~\ref{sec:non_regular_exponential_groups}.

\section{Invariant differential operators}
\label{sec:invariant_differential_operators}

Let $M$ be an exponential solvable Lie group. Its
exponential map $\exp$ is a global diffeomorphism from
its Lie algebra $\frakm$ onto $M$. In particular $M$ is
connected, simply connected. We use the fact that the
Kirillov map $\mcK$ gives a bijection from the coadjoint
orbit space $\frakm^\ast/\coAd(M)$ onto the unitary
dual $\widehat{M}$ of $M$. In particular we take the
definition of $\pi=\mcK(h)=\ind_P^M\chi_h$ via
Pukanszky / Vergne polarizations $\frakp$ at
$h\in\frakm^\ast$ for granted. Here $P$ is the unique
connected subgroup of $M$ with Lie algebra $\frakp$, and
$\chi_h(\exp X)=e^{i\,\langle h,X\rangle}$ the character
of $P$ with differential $i\,h\,|\,\frakp$. These results
can be found in Chapters~4 and~6 of~\cite{Bern_et_alii},
and Chapter~1 of~\cite{LepLud}. Mostly we shall regard
$\mcK$ as a map from $\frakm^\ast$ onto~$\widehat{M}$.
If $\mu$ is a multiplier as in Observation~\ref{cFm_obs:%
regularity_for_multipliers}, then the Kirillov
parametrization allows us to regard $\mu$ as a function
on an $\coAd(M)$-invariant subset $\Omega$ of $\frakm^\ast$
rather than on a subset $\Omega$ of $\widehat{M}$.\\\\
The following remarks apply to arbitrary Lie groups $M$.
If $\pi$ is a strongly continuous representation of $M$ in
a Banach space $E$, then the infinitesimal representation
$d\pi$ of its Lie algebra $\frakm$ is defined by
\[d\pi(X)\varphi=\frac{d}{dt}_{|t=0}\;\pi(\exp tX)\varphi\]
on the subspace $E^\infty$ of $\mcC^\infty$-vectors for $\pi$.
Dixmier and Malliavin proved that $E^{\infty}$ coincides with
the G\r{a}rding space (the dense subspace generated by vectors
of the form $\pi(a)\varphi$ with $a\in\Ccinfty{M}$ and
$\varphi\in\frakH_\rho$), see Theorem~3.3 of~\cite{DixMal}.
The representation $d\pi$ can be extended to the universal
enveloping algebra $\mcU(\frakm)$ of the complexification
$\frakm_{\mC}$ of $\frakm$. If $V\ast a=d\lambda(V)a$ denotes
the representation of $\mcU(\frakm)$ on $\Ccinfty{M}$
obtained by differentiating the left regular representation
of $M$ in $L^1(M)$, then the crucial equality
\begin{equation}\label{cFm_equ:crucial_equality}
\pi(V\ast a)=d\pi(V)\,\pi(a)
\end{equation}
holds for all $V\in\mcU(\frakm)$ and $a\in\Ccinfty{M}$.
The symmetrization map
\[\beta(X_1\mdot\;\ldots\;\mdot X_r)=\frac{1}{r!}
\sum_{\sigma\in S_r} X_{\sigma(1)}\mdot\;\ldots\;
\mdot X_{\sigma(r)}\]
gives an $\Ad(M)$-equivariant, linear isomorphism from
the symmetric algebra $\mcS(\frakm)$ of $\frakm_{\mC}$
onto $\mcU(\frakm)$, see e.g.\ Chapter~3.3 of~\cite{CorGre}.
In particular $\beta$ maps the subspace $Y(\frakm^\ast)$
of $\Ad(M)$-invariants onto the center $Z(\frakm)$
of $\mcU(\frakm)$. We identify $\mcS(\frakm)$ with
$\mcP(\frakm^\ast)$, the algebra of all complex-valued
polynomial functions on $\frakm^\ast$, by means of
the $\Ad(M)$-equivariant isomorphism of algebras
mapping $X\in\frakm$ to the linear function
$h\mapsto -i\,\langle\,h\,,X\,\rangle$
on $\frakm^\ast$.\\\\
We shall regard elements of $\mcU(\frakm)$ as
distributions on $M$ with support in~$\{e\}$, and
$\mcS(\frakm)$ as distributions on~$\frakm$ with
support in~$\{0\}$. Let $j$ be a smooth, strictly
positive function on $\frakm$. If $u$ is a
distribution on $\frakm$ with compact
support $K$, then
\[\langle\,\eta(u)\,,\,\varphi\,\rangle=
\langle\,u\,,\,j(\varphi\circ\exp)\,\rangle\]
defines a distribution $\eta(u)$ on $M$ with compact
support in $\exp(K)$. Let $U$ be the subset of all
$X\in\frakm$ such that $|\,\lambda\,|<\pi$ for
all eigenvalues $\lambda$ of $\ad(X)$. Clearly
$U$ and $V$ are open and invariant. It is known that
$\exp:U\to V$ is a diffeomorphism. Hence $\eta$ yields
a linear isomorphism from the vector space of all
distributions on $\frakm$ with compact support in $U$
onto the distributions on $M$ with compact support
in~$V$. In particular $\eta$ maps $\mcS(\frakm)$ onto
$\mcU(\frakm)$. In addition we suppose that $j$ is
$\Ad(M)$-invariant. If $u$ is $\Ad(M)$-invariant, then
$\eta(u)$ is invariant under interior automorphisms.
Thus $\eta$ maps $Y(\frakm)$ onto $Z(\frakm)$. For
$j\equiv 1$ we recover the symmetrization map $\beta$, for
\[j(X)=\det\left( \frac{1-e^{-\ad(X)}}{\ad(X)}\right)^{1/2}\]
we obtain the Duflo isomorphism $\gamma$. By means of the
character formula given in Th\'eor\`eme~II.1 and~V.2, Duflo
proved in Th\'eor\`eme~IV.1 and ~V.2 of~\cite{Duflo3} that
the restriction $\gamma:Y(\frakm)\to Z(\frakm)$ is an
isomorphism of associative algebras for all solvable and
semi-simple Lie algebras $\frakm$.

An algebraic proof of this fact
can be found in~Paragraphs~4 and 5 of~\cite{KashVer}.\\\\
In Theorem~2 of~\cite{Duflo1} Duflo proved 
$d\pi(\gamma(p))=p(h)\mdot\Id$ for all
$p\in Y(\frakm^\ast)$, $h\in\frakm^\ast$, and
$\pi=\mcK(h)$. Generalizing this property of the
pair $(\gamma(p),p)$ we state

\begin{defn}\label{cFm_defn:Duflo_pair}
Let $\Omega$ be an $\coAd(M)$-invariant subset of
$\frakm^\ast$, $W\in\mcU(\frakm)$, and
$p\in\mcP(\frakm^\ast)$. We say that $(W,p)$ is
a Duflo pair \wrt $\Omega$ if
\begin{equation}\label{cFm_equ:Duflo_pair}
d\pi(W)=p(h)\mdot\Id
\end{equation}
for all $h\in\Omega$ and $\pi=\mcK(h)$.
\end{defn}

This equality implies that $p$ is constant on all
$\coAd(M)$-orbits contained in $\Omega$ because $\mcK$ is
constant on $\coAd(M)$-orbits. We stress that we do not assume
$W\in Z(\frakm_\mC)$ nor $p\in Y(\frakm^\ast)$. If $(W,p)$ is
a Duflo pair \wrt $\Omega$ and $a\in\Ccinfty{M}$, then it
follows from Equation~(\ref{cFm_equ:crucial_equality})
and~(\ref{cFm_equ:Duflo_pair}) that $b=W\ast a\in\Ccinfty{M}$
satisfies $\pi(b)=p(h)\,\pi(a)$ for all $h\in\Omega$
and $\pi=\mcK(h)$.

Let $\frakn$ be a coabelian, nilpotent ideal of
$\frakm$, and $\Der(\frakm,\frakn)$ the subalgebra
of all derivations $D$ of $\frakm$ such that
$D\mdot\frakm\subset\frakn$. A linear functional
$h\in\frakm^\ast$ is said to be in general position
if $h(\fraka)\neq 0$ for all non-trivial
$\Der(\frakm,\frakn)$-invariant ideals $\fraka$
of $[\frakm,\frakm]$. We define $X_0=X_0(\frakm,\frakn)$
as the set of all $h\in\frakm^\ast$ in general
position satisfying the stabilizer condition
$\frakm=\frakm_h+\frakn$. Clearly $X_0$ is
$\coAd(M)$-invariant and saturated in the sense
that $X_0=X_0+[\frakm,\frakm]^\bot$. Further $X_0$
contains all $G$-orbits $X=\coAd(G)\tilde{f}$  as
in Section~\ref{sec:introduction}. We point out
that $X_0\neq\emptyset$ implies
$\frakz\frakn\subset\frakz\frakm$: If $h\in X_0$,
then $[\frakm,\frakz\frakn]=[\frakm_h,\frakz\frakn]$
is a $\Der(\frakm,\frakn)$-invariant subspace of
$\ker h$, and hence $[\frakm,\frakz\frakn]=0$.
If in addition $\frakn$ is the nilradical (i.e.~the
largest nilpotent ideal) of $\frakm$, then
$\frakz\frakn=\frakz\frakm$.\\\\
Let $\Omega_0$ be the set of all $h\in\frakm^\ast$ such
that $h'$ is in the closure of $X_0'$ in $\frakn^\ast$.
Our interest lies in the subalgebra $I(X_0)$
of all $p\in\mcS(\frakm)$ such that $p$ is
$\coAd(M)$-invariant on $\Omega_0$.\\\\
A trivial extension $\frakn=\dot{\frakn}\times\fraka$
of a Lie algebra $\dot{\frakn}$ is a direct product
with a commutative one. In particular a trivial
extension of a (k+1)-step nilpotent filiform
algebra contains
\begin{equation*}
\frakn\underset{1}{\supset}\frakc\underset{1}{\supset}
C^1\frakn+\frakz\frakn\underset{d-1}{\supset}C^1\frakn
\underset{1}{\supset}\ldots\underset{1}{\supset}
C^k\frakn\underset{1}{\supset}\{0\}
\end{equation*}
as a descending series of ideals. Here $\frakc$ is
commutative and $d=\dim\frakz\frakn$. For $k=d=1$ this
is the 3-dimensional Heisenberg algebra. If $k\ge 2$,
then $\frakc$ is the centralizer of $C^1\frakn$
in $\frakn$, and hence a characteristic ideal, too.

\begin{thm}\label{cFm_thm:existence_of_Duflo_pairs}
Let $\frakm$ be a non-nilpotent, exponential solvable
Lie algebra. If its nilradical $\frakn$ is
\begin{enumerate}
\item a trivial extension of a filiform algebra of
arbitrary dimension,
\item a $5$-dimensional Heisenberg algebra,
\item a trivial extension of $\frakg_{5,2}$,
\end{enumerate}
then $X_0$ is a non-empty, algebraic subset of $\frakm^\ast$
and the following conditions are satisfied:
\begin{enumerateroman}
\item There exist realvalued polynomial functions
$\Gamma=(\Gamma_1,\ldots,\Gamma_k)$ on $\frakm^\ast$
which are $\Ad(M)$-invariant on $X_0$ and induce
a homeomorphism $\overline{\Gamma}$ from $X_0/\coAd(M)$
onto an open subset $W$ of $\mR^k$ admitting a rational
inverse.
\item For every critical $g\in\Omega_0\setminus%
\overline{X}_0$ there exists a Duflo-pair $(W,p)$
\wrt $\Omega_0$ such that $p$ has constant term $0$ 
and $p(g)\neq 0$.
\end{enumerateroman}
 If $\frakn$ is commutative
or a trivial extension of $\frakg_{5,4}$ or
$\frakg_{5,6}$, then necessarily $X_0=\emptyset$.
If $\frakn$ is a central extension of $\frakg_{5,3}$, then the
situation is slightly more complicated: For certain
singular $g\in\Omega_0\setminus X_0$ one cannot
find $(W,p)$ such that $p(g)\neq 0$.
\end{thm}

Note that Theorem~\ref{cFm_thm:existence_of_Duflo_pairs}
covers all nilpotent Lie algebras $\frakn$ of
dimension~$\le 5$. We omit the details of its proof,
which is a case by case study. The essential steps are:
First the possible stabilizers $\frakm$ with a given
nilradical $\frakn$ are to be determined. In each case
we pick out representatives $f$ for the $\coAd(M)$-orbits
in~$X_0$. From the coadjoint representation $\coAd(M)f$
we can read off non-trivial polynomial functions $p$
on $\frakm^\ast$ which are constant on all
$\coAd(M)$-orbits in $X_0$. Finally one has to
compute suitable elements $W\in\mcU(\frakm)$. Here
the choice $W=\gamma(p)$ is natural, but not compulsory.


\section{Central Fourier multipliers}
\label{sec:central_Fourier_multipliers}

Denote by $\frakz\frakm$ the center of the Lie algebra
$\frakm$ of $M$. Let $l=\dim\frakz\frakm$ and
$k=\dim\frakm/\frakz\frakm$. We choose $b_1,\ldots,b_k$
in $\frakm$ whose canonical images in $\frakm/\frakz$ form
a Malcev basis of $\frakm/\frakz\frakm$. In particular
$\mcB=\{b_1,\ldots,b_k\}$ is a coexponential basis for
$\frakz\frakm$ in $\frakm$: If $Z(M)$ is the center
of $M$, $q:M\onto M/Z(M)$ the quotient map, and
$\Phi_1(x)=\exp(x_1b_1)\,\mdot\,\ldots\,\mdot\,\exp(x_kb_k)$,
then $q\circ\Phi_1$ is a diffeomorphism from $\mR^k$ onto
$M/Z(M)$. Equivalently, $\Phi(x,z)=\Phi_1(x)\exp(z)$ is a
global diffeomorphism from $\mR^k\times\frakz\frakm$ onto
$M$. This is a canonical coordinate system of the second
kind. If $c$ is a function on $M$, then by abuse of
notation $c\circ\Phi$ is again denoted by $c$.\\\\
We require partial Fourier transforms \wrt to the central
variable: If $c\in\LM$, then $z\mapsto f(x,z)$ is in
$L^1(\mR^l)$ for almost all $x\in\mR^k$ by Fubini's
theorem. For these $x$ and all $\xi\in\frakz\frakm^\ast$
we define
\[\widehat{c}(x,\xi)=\int_{\frakz\frakm}\;c(x,z)
e^{-i\langle\xi,z\rangle}\;dz\;.\]
For fixed $\xi\in\frakz\frakm^\ast$ the function
$x\mapsto\widehat{c}(x,\xi)$ is in $L^1(\mR^k)$.

\begin{defn}\label{cFm_defn:central_Fourier_multiplier}
A complex-valued function $\psi$ on $\frakz\frakm^\ast$ is
called a central Fourier multiplier if for all $a$ in
$\Ccinfty{M}$ there exists a (unique) smooth function $c$
in $\LM$ such that $\widehat{c}(x,\xi)=\psi(\xi)\,
\widehat{a}(x,\xi)$ holds for all $x\in\mR^k$ and
$\xi\in\frakz\frakm^\ast$.
\end{defn}

If $\psi$ is a central Fourier multiplier, so is the
function $\xi\mapsto\psi(-\xi)$. Note that $h\mapsto%
\psi(h\,|\,\frakz\frakm)$ defines an $\Ad(M)$-invariant
function on $\frakm^\ast$, and hence a function on~%
$\frakm^\ast/\coAd(M)\cong\widehat{M}$. A first
consequence is

\begin{lem}\label{cFm_lem:central_multipliers_and_%
representations}
If $c,a\in\LM$ such that $\widehat{c}(x,\xi)=\psi(-\xi)\,
\widehat{a}(x,\xi)$ for almost all $x$ and all $\xi$, then
$\pi(c)=\psi(h\,|\,\frakz\frakm)\,\pi(a)$ for all
$h\in\frakm^\ast$ and $\pi=\mcK(h)$.
\end{lem}
\begin{proof}
Let $\frakp$ be a Pukanszky polarization at $h\in\frakm^\ast$
so that $\pi=\ind_P^M\chi_h$. Since $\frakz\frakm\subset\frakp$,
it follows $\pi(\exp z)=e^{i\langle h,z\rangle}$ for all
$z\in\frakz\frakm$. The modular function $\Delta_{M,Z}$ is
trivial so that Weil's formula gives $\int_M c(m)\,dm=
\int_{\mR^k}\int_{\frakz\frakm} c(x,z)\,dzdx$ for $c\in\LM$.
Here $dm$ denotes the Haar measure of $M$, and $dx$ and $dz$
denote the Lebesgue measures on $\mR^k$ and $\frakz\frakm$
respectively. Now we obtain
\begin{align*}
\pi(c)\varphi&=\int_{\mR^k}\int_{\frakz\frakm}\;c(x,z)\,
\pi(\Phi_1(x))\pi(\exp z)\varphi\;dz\;dx\\
&=\int_{\mR^k}\;\widehat{c}(x,-h\,|\,\frakz\frakm)\,
\pi(\Phi_1(x))\varphi\;dx\\
&=\psi(h\,|\,\frakz\frakm)\;\pi(a)\varphi
\end{align*}
for every $\varphi$ in the representation space
of $\pi$.\qed
\end{proof}

The next definition is motivated by concrete applications
in the investigation of primitive $\ast$-regularity for
exponential Lie groups.

\begin{defn}\label{cFm_defn:separating_triples}
Let $X\subset\Omega$ be $\coAd(M)$-invariant subsets of
$\frakm^\ast$. We say that $(W,p,\psi)$ is a separating
triple for $X$ in $\Omega$ if
\begin{enumerateroman}
\item $\;(W,p)$ is a Duflo pair \wrt $\Omega$,
\item $\;\psi$ is a central Fourier multiplier,
\item $\;h\in\overline{X}$ if and only if $h\in\Omega$ and 
$p(h)=\psi(h\,|\,\frakz\frakm)$.
\end{enumerateroman}
\end{defn}

Condition \textit{(iii)} states that $p,\psi$ characterize the
closure of $X$ in $\Omega$, which is the closure of $X$ in
$\frakm^\ast$ if $\Omega$ is closed. Several variants of this
definition are possible: For example, one might want to consider
(finite) sets of triples $(W_\nu,p_\nu,\psi_\nu)$ such that
$h\in\overline{X}$ if and only if
$p_\nu(h)=\psi_\nu(h\,|\,\frakz\frakm)$ for all $\nu$.
The existence of separating triples is a strong assumption
making it easy to prove

\begin{thm}\label{cFm_thm:separating_triples}
If $(W,p,\psi)$ is a separating triple for $X$ in $\Omega$,
then $g\in\Omega\setminus\overline{X}$ implies
$$\bigcap\limits_{h\in X}\ker_{\LM}\mcK(h)\;\not\subset
\;\ker_{\LM}\mcK(g)\;.$$
\end{thm}
\begin{proof}
Let $\rho=\mcK(g)$. Since $\Ccinfty{M}$ is dense in $\LM$, there
exists some $a\in\Ccinfty{M}$ such that $\rho(a)\neq 0$. Define
$b=W\ast a$. Since $\psi$ is a Fourier multiplier, there exists
a unique smooth function $c\in\LM$ such that $\widehat{c}(x,\xi)=
\psi(-\xi)\,\widehat{a}(x,\xi)$. Now $b-c$ solves the problem:
Since $g\in\Omega\setminus\overline{X}$, we obtain
$$\rho(c)=\psi(\,g\,|\,\frakz\frakm\,)\,\rho(a)\neq
p(g)\,\rho(a)=\rho(b)$$
because $(W,p)$ is a Duflo pair. Furthermore, if $h\in X$ and
$\pi=\mcK(h)$, then
$$\pi(c)=\psi(\,h\,|\,\frakz\frakm\,)\,\pi(a)=
p(h)\,\pi(a)=\pi(b)$$
by Lemma~\ref{cFm_lem:central_multipliers_and_representations}.
This proves our proposition.\qed
\end{proof}
In this proof we only used the fact that
$\widehat{c}(x,h\,|\,\frakz\frakm)=\psi(h\,|\,\frakz\frakm)
\;\widehat{a}(x,h\,|\,\frakz\frakm)$ holds for all
$h\in\Omega$, instead of the full multiplier property
of~$\psi$. The continuity of $\psi$ on the closure of
$\{h\,|\,\frakz\frakm:h\in\Omega\}$ in $\frakz\frakm^\ast$
is necessary for it. If we regard $\psi$ as a function on
$\frakm^\ast$, then $p-\psi=0$ on $X$ and $\neq 0$ in $g$
so that the preceding proposition can be viewed as a
special case of the considerations in Observation~%
\ref{cFm_obs:regularity_for_multipliers}.

\section{More about central Fourier multipliers}
\label{sec:more_about_Fourier_multipliers}

As before we use a coexponential basis $\mcB$ for
$\frakz\frakm$ in $\frakm$ to define coordinates of
the second kind for $M$. We suppress the coordinate
diffeomorphism $\Phi$. Recall that $l=\dim\frakz\frakm$
and $k=\dim\frakm/\frakz\frakm$. To begin with we
introduce a subspace~$\mcQ$ of~$\LM$ which will play
a decisive role in our discussion of central Fourier
multipliers.
\begin{defn}\label{cFm_defn:subspace_Q_of_LM}
Let $\mcQ$ denote the vector space of all smooth
functions $a$ on $M$ such that
\begin{enumerate}
\item there is a compact subset $L$ of $\mR^k$ such
that $a(x,z)=0$ whenever $x\not\in L$,
\item there exists some $r_0>0$ such that for all
multi-indices $\alpha\in\mN^k$ and $\beta\in\mN^l$
the functions
\[(x,z)\mapsto (1+|z|)^{l+r_0}\;%
(D_x^\alpha D_z^\beta a)\,(x,z)\]
vanish at infinity.
\end{enumerate}
\end{defn}
If $a\in\mcQ$, then the functions $(1+|z|)^{n+r_0}%
(D_x^\alpha D_z^\beta a)$ are bounded. Alternatively
this property could have been used for a different
definition of $\mcQ$. Another possibility is to allow
the exponent $r_0$ to depend on the multi-indices
$\alpha$ and $\beta$. These alternate subspaces serve
just as well in many respects in the context of central
Fourier multipliers.\\\\
Note that $\Ccinfty{M}\subset\mcQ\subset\LM$ has the
following nice properties: The definition of $\mcQ$ does
not depend on the choice of the coexponential basis for
$\frakz\frakm$ in $\frakm$. \ifthenelse{\boolean{details}}
{We provide a short proof of this fact. Let $\mcB$ and
$\mcB'$ be coexponential bases for $\frakz\frakm$ in
$\frakm$ which define two coordinate diffeomorphisms
$\Phi$ and $\Phi'$. Let $T_1$ be the diffeomorphism on
$\mR^k$ such that $q\circ\Phi_1=(q\circ\Phi_1')\circ T_1$
and let $T$ be the diffeomorphism on
$\mR^k\times\frakz\frakm$ such that $\Phi=\Phi'\circ T$.
It follows easily that
\[T(x,z)=(T_1(x),z+\lambda(x))\]
with the smooth function $\lambda:\mR^k\to\frakz\frakm$
defined by
\[\exp(\lambda(x))=\Phi_1'(T_1(x))^{-1}\;\Phi_1(x)\;.\]
By induction we obtain
\[D_x^\alpha D_z^\beta(a\circ\Phi)(x,z)=\sum\limits_{\mu,\nu}
\;g_{\mu,\nu}(x)\;(D_{x'}^\mu D_{z'}^\nu(a\circ\Phi'))
(T_1(x),z+\lambda(x))\;.\]
Since $\lambda$ is bounded on $K$, we get
\[(1+|z|)^r\le C(1+|z+\lambda(x)|)^r\]
for some $C>0$. Since the smooth functions $g_{\mu,\nu}$ are
also bounded on $K$, we see that $a\circ\Phi'\in\mcQ'$
implies $a\circ\Phi\in\mcQ$.\\\\}{}
Clearly $\mcQ$ is a dense $\ast$-subalgebra of $\LM$ and a
$\lambda(M)$-invariant subspace where $\lambda$ denotes the
left-regular representation of $M$ in $\LM$. Furthermore
$\mcQ$ is contained in the subspace $\LM^{\infty,\lambda}$
of $\mcC^{\infty}$-vectors of $\lambda$. In particular
$\mcU(\frakm_{\mC})$ acts on $\mcQ$.\\\\
The role of $\mcQ$ in the context of central Fourier
multiplier problems is as follows: Assume that the multiplier
$\psi$ on $\frakz\frakm^\ast$ is a continuous function of
polynomial growth, and hence defines a tempered distribution.
Let $u_\psi\in\mcS'$ such that $\widehat{u_\psi}=\psi$.
Convolution with $u_\psi$ \wrt the central variable defines
a linear operator $T_\psi a=u_\psi\ast a$. In the following
we will give sufficient conditions on~$\psi$ (or~$u_\psi$)
which guarantee that $T_\psi a$ is well-defined and
in~$\mcQ$ for all~$a\in\mcQ$. From
\[(T_\psi a)\widehat{\;}\,(x,\xi)=(u_\psi\ast a)
\widehat{\;}\,(x,\xi)=\psi(\xi)\,\widehat{a}(x,\xi)\]
it will follow that $c=T_\psi a\in\mcQ$ is a solution of
the multiplier problem given by~$\psi$ and $a\in\mcQ$.
There is a twofold reason for calling $T_\psi$ a multiplier
operator: on the one hand it is multiplication by $\psi$
on the Fourier transform side, on the other hand it holds
$T_\psi(a\ast b)=(T_\psi a)\ast b$ and $(T_\psi a)^\ast
\ast b=a^\ast\ast (T_{\psi^\ast}b)$ so that
$T_\psi\in\mcQ^b$ in the spirit of Section~3
of~\cite{Lep1}.\\\\
As a starting point we choose the following well-known
result of Fourier analysis in $\mR^n$. Here one should
think of $\mR^n$ as a subspace of $\frakz\frakm$.

\begin{lem}
If $\psi\in\mcC^{n+1}(\mR^n)$ such that $D_\xi^\gamma\psi
\in L^1(\mR^n)$ for all $|\gamma|\le n+1$, then the inverse
Fourier transform $k$ of $\psi$ is a continuous function such
that $|k(z)|\le C|z|^{-(n+1)}$ for all $z\in\mR^n$, for some
$C>0$. Clearly $\langle u_\psi,\varphi\rangle=\int_{\mR^n}
k(y)\varphi(y)\,dy$.
\end{lem}

In order to obtain similar results for functions
$\psi\in\Cinfty{\mR^n\setminus\{0\}}$ which are not
differentiable in $\xi=0$, we consider dyadic decompositions
on the Fourier transform side. These ideas originated in the
work of Bernstein, Littlewood, and Paley. Up to minor
modifications, the considerations leading to the proof of
Proposition~\ref{cFm_prop:existence_of_k} can be found
on~pp.~241--246 of Stein~\cite{Stein1}.\\\\
Let $\eta\ge 0$ be in $\Ccinfty{\mR^n}$ such that
$\eta(\xi)=1$ for $|\xi|\le 1$ and $\eta(\xi)=0$ for
$|\xi|\ge 2$. Define $\delta(\xi)=\eta(\xi)-\eta(2\xi)$
so that $\supp(\delta)$ is contained in the spherical shell
$R=R(1/2,2)=\overline{B}(0,2)\setminus B(0,1/2)$. Further
we set $\delta_j(\xi)=\delta(2^{-j}\xi)$ for $j\in\mZ$ so
that $\supp(\delta_j)\subset R_j=2^jR$. For $\xi\neq 0$ we
observe that
\[\sum_{j=-l}^l\delta_j(\xi)%
=\sum_{j=-l}^l\left(\eta(2^{-j}\xi)-\eta(2^{-(j-1)}\xi)\right)
=\eta(2^{-l}\xi)-\eta(2^{l+1}\xi)\to 1\]
for $l\to+\infty$. Furthermore the series
$\sum_{j\in\mZ}\delta_j$ converges to $1$ in the sense of
tempered distributions because it is uniformly bounded by~$1$
and converges pointwise (it is locally finite on
$\mR^n\setminus\{0\}\,$).\\\\
In order to prepare the proof of Proposition~%
\ref{cFm_prop:existence_of_k} we
\ifthenelse{\boolean{details}}{verify}{state} estimates for
the cutoffs $\psi_j=\psi\delta_j$ of $\psi$. Note that
$\psi=\sum_{j\in\mZ}\psi_j$ converges in the sense of
tempered distributions. \ifthenelse{\not\boolean{details}}
{We omit proofs here.}

\begin{lem}
Let $r\ge 0$ and $\psi\in\Cinfty{\mR^n\setminus\{0\}}$
such that for every multi-index $\gamma$ there exists some
constant $A_{\gamma}>0$ such that
\[|(D_{\xi}^{\gamma}\psi)\,(\xi)|\le A_{\gamma}\;
|\xi|^{r-|\gamma|}\]
for all $\xi\neq 0$. Then there exist $A'_{\gamma}>0$ such
that $|(D_{\xi}^{\gamma}\psi_j)\,(\xi)| \le
A'_{\gamma}\;|\xi|^{r-|\gamma|}$ for $\xi\neq 0$. The new
constants $A'_\gamma$ depend on $|D_\xi^\nu\delta|_\infty$
and $A_{\nu}$ for $\nu\le\gamma$, but not on~$j$.
\end{lem}
\ifthenelse{\boolean{details}}{
\begin{proof}
First we observe that $(\partial_\xi^\nu\delta_j)(\xi)=
2^{-j|\nu|}\,(\partial_\xi^\nu\delta)(2^{-j}\xi)$ implies
$|D_\xi^\nu\delta_j|_{\infty}\le
2^{-j|\nu|}|D_\xi^\nu\delta|_\infty$. Note that
$\supp(\psi_j) \subset R_j$. If $\xi\in R_j$, then
$|\xi|\ge 2^{j-1}$ so that $2^{-j|\nu|}\le
2^{|\nu|}\,|\xi|^{-|\nu|}$. Furthermore the Leibniz
rule yields $D_\xi^{\gamma}\psi_j=\sum_{\nu\le\gamma}
{\gamma\choose\nu}(D_\xi^\nu\psi)(D_\xi^{\gamma-\nu}\delta_j)$.
Now we get
\begin{align*}
|(D_\xi^\gamma\psi_j)(\xi)|\; &\le\;\sum\limits_{\nu\le\gamma}
\;{\gamma\choose\nu}\;|(D_\xi^\nu\psi)(\xi)|\;
|(D_\xi^{\gamma-\nu}\delta_j)(\xi)|\\
& \le\;\sum\limits_{\nu\le\gamma}\;{\gamma\choose\nu}\;A_\nu\;
|\xi|^{r-|\nu|}\;2^{-j|\gamma-\nu|}\;
|D_\xi^{\gamma-\nu}\delta|_\infty\\
& \le\;\left(\;\sum\limits_{\nu\le\gamma}{\gamma\choose\nu}\;
A_{\nu}\;2^{|\gamma-\nu|}\;|D_\xi^{\gamma-\nu}\delta|_\infty
\;\right)\;|\xi|^{r-|\gamma|}
\end{align*}
which gives the definition of $A'_{\gamma}$.\qed
\end{proof}
}

Furthermore we have

\begin{lem}\label{cFm_lem:estimates_for_cutoffs}
Assume that $|(D_\xi^\gamma\psi_j)(\xi)|\le A'_\gamma\;
|\xi|^{r-|\gamma|}$ for $\xi\neq 0$. Then it follows
\[|D_\xi^\gamma(\xi^\beta\psi_j)(\xi)|\le A''_\gamma\;
|\xi|^{r+|\beta|-|\gamma|}\]
for $\xi\neq 0$ where $A''_\gamma$ depends on $A'_\nu$
for $\nu\le\gamma$.
\end{lem}

\ifthenelse{\boolean{details}}{
\begin{proof}
The Leibniz rule implies
\begin{align*}
|D_\xi^\gamma(\xi^\beta\psi_j)(\xi)|\;&\le\;
\sum\limits_{\nu\le\min\{\beta,\gamma\}}\;
{\gamma\choose\nu}\;|\xi|^{|\beta|-|\nu|}\;
|(D_\xi^{\gamma-\nu}\psi_j)(\xi)|\\
&\le\;\left(\;\sum\limits_{\nu\le\gamma}\;
{\gamma\choose\nu}\;A'_{\gamma-\nu} \;\right)\;
|\xi|^{r+|\beta|-|\gamma|}
\end{align*}
which proves the existence of $A''_{\gamma}$.\qed
\end{proof}
}

Proposition \ref{cFm_prop:existence_of_k} relies on the
following two estimates involving the geometric series.
\begin{lem}\label{cFm_lem:estimate_geometric_series}
If $m,x>0$ are real, then
\[\sum\limits_{2^j\le x^{-1}}2^{jm}\;\le\;2x^{-m}\quad
\text{and}\quad\sum\limits_{2^j>x^{-1}}2^{-jm}\;\le\;2x^m\]
where $j\in\mZ$.
\end{lem}

\ifthenelse{\boolean{details}}{
\begin{proof}
First of all $2^j\le x^{-1}$ iff $j\log 2\le -\log x$ iff
$j\le -\log x/\log 2$. Let $a=\log x/\log2$ and
$[a]=\min\{k\in\mN:a\le k\}$. Now it follows
\[\sum\limits_{2^j\le x^{-1}}2^{jm}\;=\;
\sum\limits_{j\le -a}2^{-jm}\;=\;
\sum\limits_{j=[a]}^{+\infty}2^{-jm}\;=\;
2^{-m[a]}\sum\limits_{j=0}^{+\infty}2^{-jm}\;\le\;2x^{-m}\]
because $a\le[a]$ implies $2^{-m[a]}\le 2^{-ma}=
\exp(-ma\log 2)=x^{-m}$. Similarly for the second estimate:
Clearly $2^j>x^{-1}$ iff $j>a=-\log x/\log 2$. Let
$[a]=\min\{k\in\mN:a<k\}$. Then we obtain
\[\sum\limits_{2^j>x^{-1}}2^{-jm}\;=\;\sum\limits_{j=[a]}
^{+\infty}2^{-jm}\;=\;2^{-m[a]}\sum\limits_{j=0}^{+\infty}
2^{-jm}\;\le\;2x^m\]
because $a<[a]$ implies $2^{-m[a]}\le 2^{-ma}=
\exp(-ma\log 2)=x^m$.\qed
\end{proof}
}

Now we are able to establish the validity of

\begin{prop}\label{cFm_prop:existence_of_k}
Let $r\ge 0$ and $\psi\in\Cinfty{\mR^n\setminus\{0\}}$
such that $|(D_\xi^\gamma\psi)(\xi)|\le A_\gamma\;%
|\xi|^{r-|\gamma|}$ for all $\xi\neq 0$. Since $\psi$
defines a tempered distribution, there exists a
distribution $u_\psi\in\mcS'(\mR^n)$ such that
$\widehat{u_\psi}=\psi$. Now it follows that there
is a function $k\in\Cinfty{\mR^n\setminus\{0\}}$
such that
\begin{equation}\label{cFm_equ:estimates_for_k}
|(D_z^\beta k)(z)|\le C_\beta\;|z|^{-(n+r+|\beta|)}
\end{equation}
for all $z\neq 0$ and
\begin{equation}\label{cFm_equ:k_defines_u}
\langle\,u_\psi\,,\,\varphi\,\rangle
=\int_{\mR^n}k(z)\varphi(z)\,dz
\end{equation}
for all $\varphi\in\mcS(\mR^n)$ such that $\supp(\varphi)
\subset\mR\setminus\{0\}$. Furthermore $u_\psi$ has
finite order $\le\min\{q\in\mN:n/2+r<q\}$.
\end{prop}

\begin{proof}
Let
\[k_j(z)=\psi_j^{\#}(z)=(2\pi)^{-n}\int_{\mR^n}\psi_j(\xi)
e^{i\xi z}d\xi\]
be the inverse Fourier transform of $\psi_j$. Since
$\psi=\sum_{j\in\mZ}\,\psi_j$, it follows that
$u=\sum_{j\in\mZ}\,k_j$ is convergent in the sense of
tempered distributions because  Fourier transformation is
continuous \wrt the topology of $\mcS'(\mR^n)$. We shall
estimate $\sum_{j\in\mZ}\,|(D_z^\beta k_j)(z)|$  for
$z\neq 0\,$: Since $2^{j-1}\le|\xi|\le 2^{j+1}$ for
$\xi\in R_j$, Lemma~\ref{cFm_lem:estimates_for_cutoffs}
implies
\[|D_\xi^\gamma(\xi^\beta\psi_j)\,(\xi)|\le
A''_\gamma\;|\xi|^{r+|\beta|-|\gamma|}\le
A''_\gamma\;c_{\beta,\gamma}\;2^{j(r+|\beta|-|\gamma|)}\]
where $c_{\beta,\gamma}=\max\{2^{r+|\beta|},2^{|\gamma|}\}$.
Consequently
\begin{align*}
|z^\gamma(D_z^\beta k_j)\,(z)| &=
|(D_\xi^\gamma(\xi^\beta\psi_j))^{\#}\;(z)|
\le (2\pi)^{-n}\;|D_\xi^\gamma(\xi^\beta\psi_j)|_1\\
&\le (2\pi)^{-n}\;|D_\xi^\gamma(\xi^\beta\psi_j)|_{\infty}
\;\vol(R_j)\\
&\le(2\pi)^{-n}\;A''_\gamma\;c_{\beta,\gamma}\;\vol(R)\;
2^{j(n+r+|\beta|-|\gamma|)}
\end{align*}
where $\vol(R_j)=2^{jn}\vol(R)$ denotes the Lebesgue measure
of the shell $R_j$. The validity of this inequality for all
$|\gamma|=M$ shows that there exists some $C_{\beta,M}>0$
such that
\[|(D_z^\beta k_j)(z)|\le C_{\beta,M}\;|z|^{-M}\;
2^{j(n+r+|\beta|-M)}\]
for all $z\neq 0$. Putting $M=0$ it follows by
Lemma~\ref{cFm_lem:estimate_geometric_series} that
\[\sum\limits_{2^j\le|z|^{-1}}\;|(D_z^\beta k_j)(z)|\le
2 C_{\beta,0}\;|z|^{-(n+r+|\beta|)}\;,\]
and for $M>n+r+|\beta|$ we get
\[\sum\limits_{2^j>|z|^{-1}}\;|(D_z^\beta k_j)(z)|\le
C_{\beta,M}\;|z|^{-M} \sum\limits_{2^j>|z|^{-1}}
2^{j(n+r+|\beta|-M)}\le 2 C_{\beta,M}\;
|z|^{-(n+r+|\beta|)}\;.\]
These estimates show that $\sum_{j\in\mZ}|D_z^\beta k_j|$
is uniformly convergent on compact subsets of~%
$\mR^n\setminus\{0\}$ so that $k=\sum_{j\in\mZ}k_j$
defines a smooth function on $\mR^n\setminus\{0\}$
which satisfies~\eqref{cFm_equ:estimates_for_k} and~%
\eqref{cFm_equ:k_defines_u}. The latter assertion of
this proposition follows by means of the Plancherel
theorem.\qed
\end{proof}
In view of later applications we consider functions of
the form
\[\psi(\xi)=\xi^\beta\;\left(\,1+|\xi|^2\,\right)^{-q}\;
|\xi|^r\;\log^s|\xi|\]
for $\xi\neq 0$ where $s$ is an integer, $\beta\in\mN^n$,
and $r,q$ are real. Here $|\xi|$ denotes the Euclidean
norm of $\xi\in\mR^n$. It follows by induction that its
derivatives $D_\xi^\alpha\psi$ are $\mC$-linear combinations
of functions of the form $\xi\mapsto\xi^{\beta'}
(1+|\xi|^2)^{-q'}|\xi|^{r'}\log^{s'}|\xi|$ where
$s',\beta',q',r'$ are as above and such that
$|\beta'|-2q'+r'=|\beta|-2q+r-|\alpha|$ and
$|\beta'|+r'\ge |\beta|+r-|\alpha|$. Assume that
$2q>|\beta|+r>0$ and choose $0<\epsilon<|\beta|+r$. Now
it is easy to see that there exists some $A_\alpha>0$ such
that $|\,(\text{summand of}\;D^\alpha_\xi\psi)\,(\xi)\,%
|\le A_\alpha|\xi|^{\epsilon-\alpha}$ for $\xi\neq 0$ so
that $\psi$ meets the assumptions of Proposition~%
\ref{cFm_prop:existence_of_k}.\\\\
For (spherically symmetric) functions $\psi$ of this kind
Gelfand and Shilov computed the tempered distribution~%
$T_\psi$ explicitly using methods of complex analysis
(Cauchy's theorem and analytic continuation), see
Section~3.3 of Chapter~II of~\cite{GelShi}.\\\\
Recall that a tempered distribution $u$ which satisfies
equation~(\ref{cFm_equ:k_defines_u}) of the preceding
proposition for all $\supp(\varphi)\subset\mR^n%
\setminus\{0\}$ is almost uniquely determined: Any
difference of two such distributions has support~$\{0\}$
and is thus a linear combination of derivatives of the
Dirac delta distribution.\\\\
Conversely, assume that $k\in\mcC(\mR^n\setminus\{0\})$
has an algebraic singularity of order~$\le m$ in~$0$ (here
we choose $m\ge 0$ to be the minimal integer such that
$z\mapsto|z|^m\,|k(z)|$ is bounded in a neighborhood
of~$0$), and that $k$ has decay of order $n+r$ at infinity
(there exists some $C>0$ such that $|k(z)|\le C|z|^{-(n+r)}$
for all $|z|\ge 1$). In particular $k\in L^1(\mR^n)$ and
for $0<r_0<r$ the function $z\mapsto |z|^{n+r_0}\,|k(z)|$
vanishes at infinity. By regularization of the divergent
integral $\int_{\mR^n}k(y)\varphi(y)\,dy$ we can now
define a tempered distribution $u$ on $\mR^n$: Let
$B=\overline{B}(0,1)$ be the closed ball of radius~1
around~0. For $m\ge 0$ and $\varphi\in\mcS(\mR^n)$ let
\[(P_{m-1}\,\varphi)(y)=\sum\limits_{|\nu|\le m-1}
\frac{1}{\nu!}\,(\partial^\nu\varphi)(0)\,y^\nu\]
denote its Taylor polynomial of order $m-1$ in $0$, and
\[(R_m\varphi)(y)=\varphi(y)-(P_{m-1}\,\varphi)(y)=
\sum\limits_{|\nu|=m}\frac{1}{\nu!}\,(\partial^\nu\varphi)
(\vartheta y)\,y^\nu\]
the remainder term, where $0\le\vartheta\le 1$ is chosen
suitably depending on~$y$. Clearly
\[\langle u,\varphi\rangle=\int_B k(y)\,(R_m\varphi)(y)\;dy+
\int_{\mR^n\setminus B}k(y)\varphi(y)\;dy\]
defines a tempered distribution $u$ satisfying equation~%
(\ref{cFm_equ:k_defines_u}) of Proposition~\ref{cFm_prop:%
existence_of_k}. Here we use the estimate
\[|(R_m\varphi)(y)|\,\le\,\left(\sum\limits_{|\nu|=m}\,%
\frac{1}{\nu!}\; \sup\{|(D_w^\nu\varphi)(w)|:|w|\le|y|\}%
\right)\;|y|^m\;.\]
Observe that $u$ has order $\le m$ and that $u=u_1+u_2$
is a sum of a distribution $u_1\in\mcE'(\mR^n)$ of
compact support and a distribution $u_2$ given by a
continuous function $k\in\mcC(\mR^n)$ such that
$|k(z)|\le C|z|^{-(n+r)}$ for $|z|\ge 1$, for some
$r>0$. To see this we define $u_1=\chi u$ and
$u_2=(1-\chi)u$ with $\chi\in\Ccinfty{\mR^n}$ such
that $0\le\chi\le 1$, $\chi(z)=1$ for $|z|\le 1$,
and $\chi(z)=0$ for $|z|\ge 2$.\\\\
If $u\in\mcD'(\mR^n)$ is given by a function
$k\in L^p(\mR^n)$ for some $1\le p\le\infty$, or if
$u\in\mcE'(\mR^n)$, then for $r_0>0$ the distribution~$u$
extends to a continuous linear functional on the
Fr\'echet space $\mcQ(\mR^n,r_0)$ of all functions
$a\in\Ccinfty{\mR^n}$ such that
$z\mapsto|z|^{n+r_0}\,|D_z^\beta a(z)|$ vanishes at
infinity for all multi-indices $\beta$. Here the
topology of $\mcQ(\mR^n,r_0)$ is defined by the semi-norms
\[N_m(a)=\sum_{|\beta|\le m} \left|\,(1+|z|)^{n+r_0}
\,(D_z^\beta a)\,\right|_{\infty}\]
for $m\ge 0$. Note that $\mcQ(\mR^n,r_0)$ contains
$\Ccinfty{\mR^n}$ as a dense subspace and is invariant under
differentiation, translation $(\tau_za)(y)=a(y-z)$, and
reflection $\tilde{a}(y)=a(-y)$. Any $u\in\mcQ'(\mR^n,r_0)$
defines a smooth function
\[(u\ast a)(z)=\langle\,u\,,\,\tau_z\tilde{a}\,\rangle\]
on $\mR^n$ with $D_z^\beta(u\ast a)=u\ast(D_z^\beta a)$.

\begin{lem}\label{cFm_lem:convolution_with_u}
Let $r_0>0$ and $a\in\Cinfty{\mR^n}$ such that
$z\mapsto |z|^{n+r_0}\,(D_z^\beta a)(z)$ vanishes at
infinity for all multi-indices $\beta$. Assume that either
\begin{enumerateroman}
\item $u$ is given by a function $k\in\mcC(\mR^n)$ such that
$z\mapsto |z|^{n+r_0}\,|k(z)|$ vanishes at infinity,\\
\item or that $u\in\mcE'(\mR^n)$ has finite order $m$.
\end{enumerateroman}
Then $u\ast a$ is a smooth function such that $z\mapsto
|z|^{n+r_0}\,D_z^\beta(u\ast a)(z)$ vanishes at infinity
for all $\beta$. Furthermore $\widehat{u}=\psi$ is a
continuous function of polynomial growth and
$(u\ast a)\widehat{\;}\,(\xi)=\psi(\xi)\,\widehat{a}(\xi)$.
\end{lem}

\begin{proof}
By induction it suffices to prove that $z\mapsto|z|^{n+r_0}
\,(u\ast a)(z)$ vanishes at infinity. In the first case
$(u\ast a)(z)=\int_{\mR^n}k(y)a(z-y)\,dy$. From
$|z|\le |z-y|+|y|\le 2\max\{|z-y|,|y|\}$ we deduce
\begin{multline*}
|z|^{n+r_0}\;|(u\ast a)(z)|\le 2^{n+r_0}\;%
\int_{\mR^n}|y|^{n+r_0}\,|k(y)|\,|a(z-y)|\,dy\\
+ 2^{n+r_0}\;\int_{\mR^n}|k(y)|\,|z-y|^{n+r_0}\,%
|a(z-y)|\;dy\;.
\end{multline*}
We estimate the first integral on the right hand side.
Let $\epsilon>0$ be arbitrary. Choose $R>0$ such that
$|y|^{n+r_0}\,|a(y)|\le\epsilon$ and $|y|^{n+r_0}\,|k(y)|%
\le\epsilon$ for all $|y|\ge R$, and $C>0$ such that
$|y|^{n+r_0}\,|k(y)|\le C$ for all $y\in\mR^n$.
Let $B=B(0,R)$ be the open ball of radius $R$. For
$|z|\ge 2R$ we obtain
\[\int_B |y|^{n+r_0}\,|k(y)|\,|a(z-y)|\,dy\le \epsilon
\mdot C\mdot\int_{\mR^n\setminus B}|y|^{-(n+r_0)}\,dy\]
because $y\in B$ implies $|z-y|\ge R$. Furthermore
\[\int_{\mR^n\setminus B}|y|^{n+r_0}\,|k(y)|\,|a(z-y)|\,dy
\le \epsilon\mdot\int_{\mR^n}|a(y)|\,dy\;.\]
The second integral can be treated similarly. Thus
$|z|^{n+r_0}|(u\ast a)(z)|\le C'\mdot\epsilon$ for
$|z|\ge 2R$.\\\\
Finally we assume $u\in\mcE'(\mR^n)$ so that
$\supp(u)\subset K=\overline{B}(0,R)$ for $R>0$ large enough.
There exists a $C>0$ such that
\[|\langle u,a\rangle|\le C\sum_{|\nu|\le m}|\,
D_y^\nu a\,|_{K,\infty}\]
for all $a\in\Cinfty{\mR^n}$. Note that
$D_y^\nu(\tau_z\tilde{a})(y)=(-1)^{|\nu|}(D_y^\nu a)(z-y)$.
Since $|z|\le |z-y|+|y|\le 2|z-y|$ for $y\in K$ and
$|z|\ge 2R$, we get
\begin{align*}
|z|^{n+r_0}\,|\,\langle u,\tau_z\tilde{a}\rangle\,|\; &\le\;
C\sum\limits_{|\nu|\le m}|z|^{n+r_0}\,|\,
D_y^\nu(\tau_z\tilde{a})\,|_{K,\infty}\\
& \le\;C\,2^{n+r_0}\sum\limits_{|\nu|\le m}\sup\{\,
|z-y|^{n+r_0}\,|(D_y^\nu a)(z-y)|: y\in K\,\}
\end{align*}
which tends to $0$ for $|z|\to+\infty$. The second claim of
this lemma is a consequence of the Paley-Wiener theorem for
$u\in\mcE'(\mR^n)$, and trivial for $u=k\in L^1(\mR^n)$.\qed
\end{proof}

Let us return to the global situation on $M$ and resume the
discussion of Section~\ref{sec:central_Fourier_multipliers}.
Using the function space $\mcQ$ introduced in Definition~%
\ref{cFm_defn:subspace_Q_of_LM} we state a refinement of
Definition~\ref{cFm_defn:central_Fourier_multiplier}.

\begin{defn}\label{cFm_defn:central_multipliers_refined}
A continuous function $\psi$ on $\frakz\frakm^\ast$ is called
a central Fourier multiplier if for all $a\in\mcQ$ there
exists a (unique) function $c\in\mcQ$ such that
$\widehat{c}(x,\xi)=\psi(\xi)\,\widehat{a}(x,\xi)$ for
all $x$ and $\xi$. These functions form a subalgebra~$\mcM$
of~$\mcC(\frakz\frakm^\ast)$ containing the polynomial
functions and the Schwartz functions.
\end{defn}

If we interpret the solution $c$ as a function on $M$ rather
than on $\mR^k\times\frakz\frakm$, then its definition does
not depend on the choice of the coordinates furnished by a
coexponential basis~$\mcB$ for $\frakz\frakm$ in $\frakm$.
\ifthenelse{\boolean{details}}{
We give a short proof of this assertion: For any $a\in\LM$
we have
\[(a\circ\Phi)\,\widehat{\;}\,(x,\xi)=e^{i<\xi,\lambda(x)>}\;
(a\circ\Phi')\,\widehat{\;}\,(T_1(x),\xi)\;.\]
If $a,c\in\LM$ such that $(c\circ\Phi')\,\widehat{\;}\,%
(x',\xi)=\psi(\xi)\;(a\circ\Phi')\,\widehat{\;}\,(x',\xi)$,
then it follows immediately from the above that
$(c\circ\Phi)\,\widehat{\;}\,(x,\xi)=\psi(\xi)\;
(a\circ\Phi)\,\widehat{\;}\,(x,\xi)$ which yields
the assertion.\\\\}{\\\\}
Let us fix a direct sum decomposition $\frakz\frakm=\frakz
\oplus\tilde{\frakz}$ of the center of $\frakm$ and denote
the central variable by $(z,\tilde{z})$. This also gives a
decomposition $\frakz\frakm^\ast=\frakz^\ast\oplus
\tilde{\frakz}^\ast$ of the linear dual with variable
$(\xi,\tilde{\xi})$. Let us identify $\frakz$ with $\mR^n$.
Assume that $a\in\mcQ$ and $r_0>0$ such that
\[(x,z,\tilde{z})\mapsto |(z,\tilde{z})|^{n+r_0}\;
|\,(D_x^\alpha D_z^\beta D_{\tilde{z}}^{\tilde{\beta}}a)
(x,z,\tilde{z})\,|\]
vanishes at infinity. Clearly $a^{\sharp}(x,\tilde{z})(z)=
a(x,z,\tilde{z})$ defines a smooth function $a^{\sharp}:
\mR^k\times\tilde{\frakz}\to\mcQ(\frakz,r_0)$ with
$D_x^\alpha D_{\tilde{z}}^{\tilde{\beta}}a^{\sharp}
=(D_x^\alpha D_{\tilde{z}}^{\tilde{\beta}}a)^{\sharp}$.
Further any $u\in\mcQ'(\frakz,r_0)$ gives rise to a
smooth function
\[(u\ast a)(x,z,\tilde{z})=\langle\,u,\tau_z
(a^{\sharp}(x,\tilde{z})\tilde{\;})\,\rangle\]
on $M$ such that $D_x^\alpha D_z^\beta D_{\tilde{z}}^
{\tilde{\beta}}(u\ast a)=u\ast(D_x^\alpha D_z^\beta
D_{\tilde{z}}^{\tilde{\beta}}a)$. Here translation and
reflection affect only the variable~$z$.

\begin{lem}\label{cFm_lem:convolution_with_u_refined}
Assume that $u\in\mcD'(\frakz)$ is given by a function
$k\in\mcC(\frakz)$ such that $z\mapsto|z|^{n+r}\,|k(z)|$
vanishes at infinity for some $r>0$, or that
$u\in\mcE'(\frakz)$. It follows $u\ast a\in\mcQ$ for
all $a\in\mcQ$. In particular $\psi=\widehat{u}$ lies
in $\mcM$ when interpreted as a function on
$\frakz\frakm^\ast$.
\end{lem}

\begin{proof}
We know that $\psi=\widehat{u}$ is a continuous function
of polynomial growth, and that $u\ast a$ is well-defined,
smooth and of compact support in $x$-direction. Choose
$0<r_0\le r$ as above. It remains to be shown that the
derivatives of $u\ast a$ multiplied by the factor
$|(z,\tilde{z})|^{n+r_0}$ vanish at infinity. But
this follows as in the proof of Lemma~\ref{cFm_lem:%
convolution_with_u} by analogous estimates performed
uniformly in~$x$ and~$\tilde{z}$. Evidently the partial
Fourier transform of $u\ast a$ \wrt the central
variable~$(z,\tilde{z})$ satisfies
$(u\ast a)\widehat{\;}\,(x,\xi,\tilde{\xi})
=\psi(\xi)\,\widehat{a}(x,\xi,\tilde{\xi})$ which
proves $\psi\in\mcM$.\qed
\end{proof}

At last we discuss a class of functions $\psi$ which arise
naturally as central Fourier multipliers in studying the
primitive $\ast$-regularity of exponential Lie groups.
Let $\frakz\frakm=\frakz_1\oplus\ldots\oplus\frakz_{l'}\oplus
\tilde{\frakz}$ be a direct sum decomposition of the center
of $\frakm$ with a Euclidean norm on each of these subspaces,
and $Q=Q(r_1,\ldots,r_{l'})$ a complex-valued polynomial
function in $l'$ real variables. The function
\[\psi(\xi,\tilde{\xi})=|\xi_1|\,\mdot\,\ldots\,\mdot
\,|\xi_{l'}|\,\mdot\,Q(\,\log|\xi_1|,\ldots,
\log|\xi_{l'}|\,)\]
on $\frakz\frakm^\ast$ is a linear combination of products
of functions $\xi\mapsto|\xi|\,\log^s|\xi|$ defined on
one of the subspaces $\frakz_\nu$. Since the polynomial
$1+|\xi|^2$ is in $\mcM$ and $\psi_0(\xi)=(1+|\xi|^2)^{-1}
\,|\xi|\,\log^s|\xi|$ is in $\mcM$ by Proposition~%
\ref{cFm_prop:existence_of_k} and Lemma~\ref{cFm_lem:%
convolution_with_u_refined}, it follows that $\psi\in\mcM$
is a central Fourier multiplier in the sense of
Definition~\ref{cFm_defn:central_multipliers_refined}
because $\mcM$ is an associative algebra.

\section{A functional calculus for central
elements}\unboldmath\label{sec:functional_calculus}
We fix a coexponential basis for $\frakz\frakm$
in $\frakm$ as in the beginning of Section~\ref{sec:%
central_Fourier_multipliers} and work with the coordinates
of the second kind associated to it. Let $\mcQ$ be as in
Definition~\ref{cFm_defn:subspace_Q_of_LM}. If $L$ is a
compact subset of $\mR^k$ and $r_0>0$, then $\mcQ(L,r_0)$
denotes the subspace of all smooth functions $a$ on $M$
such that $a(x,z)=0$ whenever $x\not\in L$ and such that
$(x,z)\mapsto|z|^{n+r_0}\,(D_x^\alpha D_z^\beta a)(x,z)$
vanishes at infinity. As a topological vector space
$\mcQ$ is the inductive limit (convex hull) of the
Fr\'echet spaces $\mcQ(L,r_0)$. In particular $\mcQ$ is
an $(LF)$-space (but not a strict one), compare~\S 19
of~\cite{Koethe}. The definition of the topology of $\mcQ$
does not depend on the choice of the coexponential basis.
Clearly, the universal enveloping algebra $\mcU(\frakm_{\mC})$
acts on $\mcQ$ (in the natural way) as an algebra of
continuous linear operators.\\\\
Let $\mcM$ denote the algebra of central Fourier multipliers
introduced in Definition~\ref{cFm_defn:central_multipliers%
_refined}. If $\psi\in\mcM$ and $a\in\mcQ$, then $T_{\psi}\,a$
denotes the unique function in $\mcQ$ such that $(T_{\psi}\,a)
\widehat{\;}(x,\xi)=\psi(\xi)\,\widehat{a}(x,\xi)$. Note that
$(\psi,a)\mapsto T_{\psi}\,a$ defines a representation
of~$\mcM$ on~$\mcQ$. At least if $\psi$ has the form $\psi=
\widehat{u}$ with $u=u_1+u_2$, $u_1$ of compact support,
and $u_2$ given by a continuous function $k$ of growth
$|k(z)|\le C|z|^{-(n+r_0)}$, then $T_\psi$ is a continuous
operator on~$\mcQ$. The set of all multipliers $\psi$ for
which $T_{\psi}$ is continuous is a subalgebra of
$\mcM$.\\\\
If $a\in\mcQ$, $Z\in\frakz\frakm$, and $\psi(\xi)=\langle\,
\xi,Z\,\rangle$, then
\[(Z\ast a)(x,y)=\frac{d}{dt}_{|t=0}\,a\left(\,\exp(-tZ)
\mdot(x,y)\,\right)=-\langle\,\partial a(x,y),Z\,\rangle
=-(\partial_Z a)(x,y)\]
implies $(iZ\ast a)\widehat{\;\;}(x,\xi)=\langle\,\xi,Z\,
\rangle\;\widehat{a}(x,\xi)$ which proves
$T_{\psi}\,a=iZ\ast a$. Here $\partial a:M\to\Hom_{\mR}
(\frakz\frakm,\mC)$ is the derivative of $a$ \wrt the central
variable, and $\partial_Z\,a$ the directional derivative.\\\\
Since the action of $\mcM$ on $\mcQ$ commutes with the action
of $\mcU(\frakm_{\mC})$, we see that $\mcM$ extends
$\mcU(\frakz\frakm_{\mC})=\mcS(\frakz\frakm_{\mC})$. Here
we identify elements of the symmetric algebra $\mcS(\frakz%
\frakm_{\mC})$ with their symbols. In this sense we have
enlarged the features of the symmetric algebra of
$\frakz\frakm$ from polynomial functions to (certain)
functions of polynomial growth.\\\\
Finally we would like explain the heading of this section:
Let $Z\in\frakz\frakm$ be a central element. We know that
$iZ\ast\,-$ acts as a differential operator on $\mcQ\subset
\LM$ and we want to declare the notion of functions of this
operator. It follows from $(iZ\ast\,a)\widehat{\;\;}(x,\xi)=
\langle\,\xi,Z\,\rangle\;\widehat{a}(x,\xi)$ that this
operator is diagonalized by partial Fourier transformation.
Let $\psi_0:\mR\to\mC$ be a continuous function such that
$\xi\mapsto\psi(\xi)=\psi_0(\,\langle\xi,Z\rangle\,)$ is
in~$\mcM$. It is a basic idea of any definition of
$\psi_0(iZ\ast\,-)$ that $\left(\,\psi_0(iZ\ast\,-)a\,
\right)\widehat{\;\;}\,(x,\xi)= \psi_0(\,\langle\xi,Z
\rangle\,)\;\widehat{a}(x,\xi)$ should hold. Thus the
definition $\psi_0(iZ\ast\,-)a:=T_\psi\,a$ appears to be
reasonable and we have indeed established a functional
calculus for central elements.

\boldmath\section{Two non-$\ast$-regular
exponential Lie groups}\unboldmath
\label{sec:non_regular_exponential_groups}

Our aim is to prove that the following two significant
examples of non-$\ast$-regular exponential Lie groups have
the weaker property of primitive $\ast$-regularity, see
Definition~1 of~\cite{Ung1}. The results of the preceding
sections (in particular those related to separating triples
consisting of a Duflo pair $(W,p)$ and a central Fourier
multiplier $\psi$) turn out to be appropriate for this
purpose. A first example (of minimal dimension) has
already been discussed in~\cite{Ung1}. In order to prove
the primitive $\ast$-regularity of an exponential Lie
group $G$ we pursue the strategy developed in Section~5
of~\cite{Ung1}.\\\\
For the convenience of the reader we provide a brief
history of $\ast$-regularity. In~\cite[1978]{Boid1} Boidol
and Leptin initiated the investigation of the class~$[\Psi]$
of $\ast$-regular locally compact groups. Far reaching
results have been obtained in this direction. First
Boidol has characterized the $\ast$-regular ones among
all exponential Lie groups by a purely algebraic condition
on the stabilizers $\frakm=\frakg_f+\frakn$ of linear
functionals $f\in\frakg^\ast$, see Theorem~5.4
of~\cite{Boid2} and Lemma~2 of~\cite{Pog4}. More
generally Boidol has proved in~\cite{Boid3} that a
connected locally compact group is $\ast$-regular if
and only if all primitive ideals of $\CG$ are
(essentially) induced from a normal subgroup $M$
whose Haar measure has polynomial growth. In~\cite{Pog4}
Poguntke has determined the simple modules of the group
algebra $\LG$ for exponential Lie groups $G$. From this
classification he deduced that an exponential Lie group~$G$
is $\ast$-regular if and only if it is symmetric, i.e.,
$a^\ast a$ has positive spectrum for all $a\in\LG$, see
Theorem~10 of~\cite{Pog4}. A complete list of all
non-symmetric solvable Lie algebras up to dimension 6
can be found in~\cite{Pog2}. For a definition of primitive
$\ast$-regularity and $L^1$-determined ideals we refer
to~\cite{Ung1}.\\\\
As in Section~5 of~\cite{Ung1} we fix a coabelian, nilpotent
ideal~$\frakn$ (e.g.~the nilradical, i.e., the largest
nilpotent ideal) of the Lie algebra~$\frakg$ of $G$. Now
it suffices to verify the following two assertions:
\begin{enumerate}
\item Every proper quotient $\bar{G}$ of $G$ is primitive
$\ast\,$-regular.
\item If $f\in\frakg^\ast$ is in general position such that
the stabilizer $\frakm=\frakg_f+\frakn$ is a proper,
non-nilpotent ideal of $\frakg$ and if $g\in\frakg^\ast$
is critical for the orbit $\coAd(G)f$, then it follows
$$\ker_{\LG}\pi\not\subset\ker_{\LG}\rho$$
for the unitary representations $\pi=\mcK(f)$ and
$\rho=\mcK(g)$.
\end{enumerate}
We say that $f$ is in general position if $f\neq 0$
on any non-trivial ideal of $\frakg$. Here
$\frakg_f=\{X\in\frakg:[X,\frakg]\subset\ker f\}$ is the
stabilizer of $f$ \wrt~the coadjoint action of $\frakg$ on
$\frakg^\ast$. Note that the ideal $\frakm=\frakg_f+\frakn$
does not depend on the choice of the representative $f$ of
the coadjoint orbit $\coAd(G)f$. Let $\Omega$ denote the
set of all $h\in\frakg^\ast$ such that its restriction
$h'=h\,|\,\frakn$ is contained in the closure of $\coAd(G)f'$
in $\frakn^\ast$. We say that $g\in\frakg^\ast$ is critical
\wrt the orbit $X=\coAd(G)f$ if and only if $g\in\Omega
\setminus\overline{X}$.\\\\
When restricting to the stabilizer $M$ with Lie algebra
$\frakm=\frakg_f+\frakn$, the representation $\pi$ in general
position decomposes into a direct integral of irreducible
representations $\pi_s=\mcK(f_s)$ of $M$, and in the Kirillov
picture the associated coadjoint orbit $\coAd(G)f$ decomposes
into the disjoint union of the orbits $\coAd(M)f_s$. Now it is
easy to see that we can replace the second assertion by the
following equivalent one:
\begin{enumerate}
\item[(3)] Let $\frakm$ be a proper, non-nilpotent ideal of
$\frakg$ such that $\frakm\supset\frakn$. If $f\in\frakm^\ast$
is in general position such that $\frakm=\frakm_{f}+\frakn$
and if $g\in\frakm^\ast$ is critical for the orbit $\coAd(G)f$,
then the relation
\[\bigcap\limits_{s\in\mR^m}\,\ker_{\LM}\,\pi_s\not
\subset\ker_{\LM}\,\rho\]
holds for the representations $\pi_s=\mcK(f_s)$ and
$\rho=\mcK(g)$.
\end{enumerate}
At this point the results of the preceding sections come into
play. If one can prove the existence of (a finite set of)
separating triples for $X=\coAd(G)f$ in $\Omega\subset
\frakm^\ast$ in the sense of Definition~\ref{cFm_defn:%
separating_triples}, then the asserted relation for the
$L^1$-kernels of the associated irreducible representations
follows at once.\\\\
Now we delve into the details of our first example. Let $G$
be a simply connected, connected, solvable Lie group such
that the nilradical $\frakn$ of its Lie algebra $\frakg$
is a trivial extension of the five-dimensional, two-step
nilpotent Lie algebra $\frakg_{5,2}$ so that
\[\frakg\supset\frakm\supset\frakn\underset{3}{\supset}
\frakz\frakn\supset C^1\frakn\underset{2}{\supset}\{0\}\]
is a descending series of ideals of $\frakg$. Assume that
$d,e_0,\ldots,e_6$ is a basis of $\frakg$ with commutator
relations $[e_1,e_2]=e_4$, $[e_1,e_3]=e_5$, $[e_0,e_1]=-e_1$,
$[e_0,e_2]=e_2$, $[e_0,e_3]=e_3$, $[d,e_0]=-ae_6$,
$[d,e_2]=e_2$, $[d,e_3]=be_3$, $[d,e_4]=e_4$, $[d,e_5]=be_5$
where $a,b\in\mR$ and $b\neq 0$. Furthermore we
assume that $\frakm=\langle e_0,\ldots,e_6\rangle$ and that
$f\in\frakm^\ast$ is in general position such that
$\frakm=\frakm_f+\frakn$. In particular $f\neq 0$ on the
one-dimensional ideal spanned by $e_\nu$, for all
$4\le\nu\le 6$.\\\\
The algebraic structure of $\frakg$ is characterized by the
fact that the nilpotent subalgebra $\fraks=\langle d,e_0,
e_6\rangle$ acts semi-simply on the nilradical
$\frakn=\langle e_1,\ldots,e_6\rangle$ with weights
$\alpha,\gamma-\alpha$, $b\gamma-\alpha$, $\gamma$,
$b\gamma$, $0$ where $\alpha,\gamma\in\fraks^\ast$ are
linearly independent and given by $\alpha(e_0)=-1$,
$\alpha(d)=\alpha(e_6)=0$ and $\gamma(d)=1$,
$\gamma(e_0)=\gamma(e_6)=0$.
\begin{lem}\label{cFm_lem:g52_special_representative}
If $f\in\frakm^\ast$ is in general position such that the
stabilizer condition $\frakm=\frakm_f+\frakn$ is satisfied,
then there exists a representative $f$ on the orbit
$\coAd(M)f=\coAd(N)f$ such that $f_1=f_2=f_3=0$.
\end{lem}
\begin{proof}
Since $f_4\neq 0$, the equations
\begin{align*}
\coAd(\exp(we_2))f\;(e_1)&=f_1+wf_4\\
\coAd(\exp(ve_1))f\;(e_2)&=f_2-vf_4\\
\end{align*}
show that we can establish $f_1=0$ and $f_2=0$. Here we
abbreviate $f(e_\nu)$ by~$f_\nu$. Since
$\frakm=\frakm_f+\frakn$, there is some
$X=te_0+ve_1+we_2+xe_3+Z\in\frakm_f$ such that $t\neq 0$.
Now $[X,e_2]=te_2+ve_4$ and $[X,e_3]=te_3+ve_5$ implies
$0=vf_4$ and $0=tf_3+vf_5$. Since $f_4\neq 0$ and $t\neq 0$,
it follows $v=0$ and $f_3=0$.\qed
\end{proof}
In the sequel we fix $f\in\frakm^\ast$ such that $f_\nu=0$
for $1\le\nu\le 3$ and $f_\nu\neq 0$ for $4\le\nu\le 6$.
By adjusting the basis vectors $e_2,\ldots,e_6$ we can
even establish $f_\nu=1$ for $4\le\nu\le 6$. Using
coordinates of the second kind given by the diffeomorphism
$\Phi(t,v,w,x,Z)=\exp(te_0)\,\exp(ve_1)\,\exp(we_2+xe_3+Z)$
we compute
\begin{align*}
\coAd(\exp(sd)\Phi(t,v,w,x,Z))f\;(e_0)&=f_0+as-v(w+x),\\
(e_1)&=e^t\,(w+x),\\
(e_2)&=-e^{-(s+t)}\,v,\\
(e_3)&=-e^{-(bs+t)}\,v,\\
(e_4)&=e^{-s},\\
(e_5)&=e^{-bs}
\end{align*}
for the coadjoint action of $G$ on $\frakm^\ast$. These
formulas motivate the definition of the polynomials
$p_1=e_0\,e_4-e_1\,e_2$, $p_2=e_0\,e_5-e_1\,e_3$, and
$p_3=e_2\,e_5-e_3\,e_4$. Here $e_\nu$ means the linear
function $f\mapsto f(e_\nu)$ on~$\frakm^\ast$, considered
as an element of $\mcP(\frakm^\ast)$, the commutative
algebra of complex-valued polynomial functions
on~$\frakm^\ast$. Recall that $M$ acts on
$\mcP(\frakm^\ast)$ by $\Ad(m)p\,(f)=p(\coAd(m)^{-1}f)$.\\\\
Note that these three polynomial functions are constant on
the orbits $\coAd(M)f_s$ for all $f_s=\coAd(\exp(sd))f$, but
none of them is $\Ad(M)$-invariant (constant on all
$\coAd(M)$-orbits).\\\\
A first step is to determine the $\frakn^\ast$-closure
$\Omega$ of the orbit $X=\coAd(G)f$. To this end let
$r:\frakm^\ast\onto\frakn^\ast$ denote the linear
projection given by restriction and define
$\Omega=r^{-1}(r(X)\clos)$. In order to avoid
trivialities we shall suppose $b>0$.

\begin{lem}
The $\frakn^\ast$-closure $\Omega$ of $X$ is contained in the
$\coAd(M)$-invariant set of all $h\in\frakm^\ast$ such
that either ($h_4>0$, $h_5>0$, $h_6=1$, $\log h_5=b\log h_4$,
and $p_3(h)=h_2h_5-h_3h_4=0$) or ($h_4=h_5=0$ and $h_6=1$).
\end{lem}

This assertion will be verified in the course of the proof of
Lemma~\ref{cFm_lem:g52_characterize_closure_of_X}. It follows
from Lemma~\ref{cFm_lem:g52_special_representative} that not
all linear functionals in general position satisfy the
stabilizer condition $\frakm=\frakm_f+\frakn$. Furthermore
$p_3(h)=0$ for all $h\in\Omega$ so that $\Omega$ is rather
sparse.

\begin{rem}
The polynomial functions $p_1,\ldots,p_3$ are constant on
all $\coAd(M)$-orbits contained in $\Omega$ and yield
continuous functions on $\Omega/\coAd(M)$. There do not
exist any 'non-trivial' $\Ad(M)$-invariant polynomial
functions on $\frakm^\ast$.
\end{rem}
This observation applies to many other examples as well.
Retrospectively, it justifies the localization to a
certain subset~$\Omega$ of~$\frakm^\ast$ (or of
$\widehat{M}$) that we started with in Section~1.
Here 'non-trivial' means something like
$p\in\mcS(\frakm_{\mC})\setminus\mcS(\frakz\frakm_{\mC})$.\\\\
Next we describe the relevant unitary representations of $M$.
Let $f_s=\coAd(\exp(sd))f$ and $\pi_s=\mcK(f_s)$. It is easy
to see that $\frakp=\langle\,e_0,e_2,\ldots,e_6\,\rangle$ is
a Pukanszky-Vergne polarization at $f_s$ for all $s\in\mR$,
and that $\frakc=\langle\,e_1\,\rangle$ is a coexponential
subalgebra for $\frakp$ in $\frakm$. For the infinitesimal
operators of the unitary representation
$\pi_s=\ind_P^M\chi_{f_s}$ we compute
\begin{align*}
d\pi_s(e_0)&=\frac{1}{2}+i(f_0+as)+\xi\partial_\xi,\\
d\pi_s(e_1)&=-\partial_\xi,\\
d\pi_s(e_2)&=-ie^{-s}\,\xi,\\
d\pi_s(e_3)&=-ie^{-bs}\,\xi,\\
d\pi_s(e_4)&=ie^{-s},\\
d\pi_s(e_5)&=ie^{-bs}.
\end{align*}
Now let $g\in\frakm^\ast$ be such that $g_5=g_4=0$ and
($g_1\neq 0$ or $g_2\neq 0$ or $g_3\neq 0$). Then $\frakn=
\langle\,e_1,\ldots,e_5\,\rangle$ is a Pukanszky-Vergne
polarization at $g$. Further $\frakc=\langle\,e_0\,\rangle$
is a coexponential subalgebra for $\frakn$ in $\frakm$.
Hence $\rho=\ind_N^M\chi_g$ is infinitesimally given by
\begin{align*}
d\rho(e_0)&=-\partial_{\xi},\\
d\rho(e_1)&=ie^{\xi}\,g_1,\\
d\rho(e_2)&=ie^{-\xi}\,g_2,\\
d\rho(e_3)&=ie^{-\xi}\,g_3,\\
d\rho(e_4)&=d\rho(e_5)=0.
\end{align*}
The images of $p_1$ and $p_2$ in the universal enveloping
algebra $\mcU(\frakm_\mC)$ under the symmetrization map
are given by
\[W_1=\frac{1}{2}\,(e_1\,e_2+e_2\,e_1)-e_0\,e_4
\quad\text{and}\quad
W_2=\frac{1}{2}\,(e_1\,e_3+e_3\,e_1)-e_0\,e_5\]
respectively. A short computation shows that $d\pi(W_\nu)=
p_\nu(h)\mdot\Id$ holds for all $h\in\Omega$ and
$\pi=\mcK(h)$. Thus $(W_\nu,p_\nu)$ is a Duflo pair
\wrt~$\Omega$, for $\nu\in\{1,2\}$. In addition we
define the continuous functions
$\psi_1(\xi)=\xi_1\,(f_0-a\log|\xi_1|\,)$ and
$\psi_2(\xi)=\xi_2\,(f_0-\frac{a}{b}\,\log|\xi_2|\,)$ on
$\frakz\frakm^\ast$. Here we identify $\mR^3$ and
$\frakz\frakm^\ast$ via $\xi=(\xi_1,\xi_2,\xi_3)
\mapsto\xi_1e_4^\ast+\xi_2e_5^\ast+\xi_3e_6^\ast$.
Now we can prove

\begin{lem}\label{cFm_lem:g52_characterize_closure_of_X}
Assume that $f\in\frakm^\ast$ is in general position such
that $\frakm=\frakm_f+\frakn$. Let $W_\nu,p_\nu,\psi_\nu$
be defined as above. If $\Omega$ denotes the
$\frakn^\ast$-closure of the orbit $X=\coAd(G)f$, then
$\{\,(W_\nu,p_\nu,\psi_\nu):\nu=1,2\,\}$ is a set of
separating triples for~$X$ in~$\Omega$ in the sense of
Definition~\ref{cFm_defn:separating_triples}. In particular
$\psi_\nu$ is a central Fourier multiplier. If
$h\in\Omega$, then $h\in\overline{X}$ if and only if
$p_\nu(h)=\psi_\nu(h\,|\,\frakz\frakm)$ for all $\nu$.
\end{lem}
\begin{proof}
The considerations at the end of Section~\ref{sec:more_%
about_Fourier_multipliers} reveal that $\psi_1$ and
$\psi_2$ are central Fourier multipliers in the sense
of Definition~\ref{cFm_defn:central_multipliers_refined}.
By definition $p_\nu(h)=\psi_\nu(h\,|\,\frakz\frakm)$ for
all $h\in X$. The continuity of $p_\nu$ and $\psi_\nu$
yields this equality for all $h\in\overline{X}$. In order to
prove the opposite implication, we assume that $h\in\Omega$
such that $p_\nu(h)=\psi_\nu(h\,|\,\frakz\frakm)$ for
$\nu\in\{1,2\}$. In particular there exist sequences $s_n$,
$v_n$, $w_n$, $x_n$ such that $f_n'\to h'$ where
\[f_n=\coAd\left(\,\exp(s_nd)\,\Phi(0,v_n,w_n,x_n,0)
\,\right)f\;.\]
At first we suppose $h_4h_5\neq 0$. In this case
$e^{-s_n}\to h_4$ and $e^{-bs_n}\to h_5$ implies $h_4>0$,
$h_5>0$, and $\log h_5=b\log h_4$. Similarly it follows
$p_3(h)=0$, and $h_6=1$ is obvious. If we choose sequences
$s_n,\ldots,x_n$ as above, then we obtain
\begin{multline*}
e^{-s_n}\left(f_0+as_n-v_n(w_n+x_n)\right)
=\psi_1(f_n\,|\,\frakz\frakm)+f_{n1}f_{n2}\\
\to\psi_1(h\,|\,\frakz\frakm)+h_1h_2=h_0h_4
\end{multline*}
because $p_1(h)=\psi_1(h\,|\,\frakz\frakm)$. Now
$e^{-s_n}\to h_4\neq 0$ implies $f_{0n}\to h_0$ and
hence $f_n\to h\in\overline{X}$.\\\\
Next we assume $h_4=0$ or $h_5=0$. We conclude $h_4=h_5=0$
and $b>0$. Now we must distinguish several subcases. In any
case we set $x_n=0$. First we assume $h_1\neq 0$. Since
$p_\nu(h)=\psi_\nu(h\,|\,\frakz\frakm)=0$ for $\nu\in\{1,2\}$,
it follows $h_2=h_3=0$. We define $s_n=n$, $w_n=h_1$, and
\[v_n=\frac{1}{h_1}\,(f_0+as_n-h_0)\]
so that $f_n\to h$. Next we assume $h_1=0$ and ($h_2\neq 0$
or $h_3\neq 0$). In this case we choose sequences $s_n$ and
$v_n$ such that $f_n(e_\nu)\to h_\nu$ for $2\le\nu\le 5$. In
particular $s_n\to +\infty$ and $|v_n|\to +\infty$
exponentially. Further we set
$$w_n=\frac{1}{v_n}\,(f_0+as_n-h_0)\;.$$
Then we obtain $f_n\to h$. Finally we assume $h_\nu=0$ for
$1\le\nu\le 5$. We define $s_n=n$, $v_n=e^{r_n/2}$ and
$w_n=e^{-r_n/2}\,(f_0-h_0)$. These definitions imply
$f_n\to h$. This completes the proof of our lemma.\qed
\end{proof}

Observe that both polynomials $p_1$ and $p_2$ are needed
to separate points $h\in\Omega$ with $h_4=h_5=0$,
$h_1\neq0$, and ($h_2\neq 0$ or $h_3\neq 0$) from the
orbit $X=\coAd(G)f$.\\\\
As we remarked above, the fact that
$\{(W_\nu,p_\nu,\psi_\nu):\nu=1,2\}$ is a set of
separating triples for $X$ in $\Omega$ yields
\[\bigcap\limits_{s\in\mR}\;\ker_{\LM}\,\pi_s\;\not\subset\;
\ker_{\LM}\,\rho\]
for all critical $g\in\Omega$ and $\rho=\mcK(g)$. Thus
$\ker_{\CG}\pi$ is $L^1$-determined in the sense of
Definition~1 of~\cite{Ung1} for all representations $\pi$
in general position such that its stabilizer $M$ is a
non-nilpotent normal subgroup of $G$. If $M=N$, then
$\ker_{\CG}\pi$ is $L^1$-determined by Proposition~2.6
and~2.8 of \cite{Ung1}. Up to this point we have shown
that $\ker_{\CG}\pi$ is $L^1$-determined for all $\pi$
in general position.\\\\
If $\pi$ is not in general position, then we can pass to
a proper quotient~$\bar{G}$ of~$G$. For example, if
$f(e_6)=0$, then we can pass to the quotient $\bar{\frakg}
=\frakg/\langle e_6\rangle$. We assume that
$\bar{f}\in\bar{\frakg}^\ast$ is in general position such
that $\bar{\frakm}=\bar{\frakg}_{\bar{f}}+\bar{\frakn}$ is
a proper, non-nilpotent ideal of $\bar{\frakg}$. It is easy
to see that in this case $\bar{p}_1=\bar{e}_0\bar{e}_4-
\bar{e}_1\bar{e}_2$, $\bar{\psi}_1(\xi)=f_0\xi_1$ and
$\bar{p}_2=\bar{e}_0\bar{e}_5-\bar{e}_1\bar{e}_3$,
$\bar{\psi}_2(\xi)=f_0\xi_2$ form the ingredients for a set
of separating triples for $\bar{X}=\coAd(\bar{G})\bar{f}$
in $\bar{\Omega}$. As above it follows that
$\ker_{C^\ast(\bar{G})}\bar{\pi}$ is $L^1$-determined
for all representations $\bar{\pi}$ of $\bar{G}$ in
general position.\\\\
Clearly the quotients $\bar{\frakg}=\frakg/\langle e_4\rangle$
and $\bar{\frakg}=\frakg/\langle e_5\rangle$ can be treated
similarly: In these cases we get by on one separating triple.
Choose $\bar{p}=\bar{e}_0\bar{e}_5-\bar{e}_1\bar{e}_3$,
$\bar{\psi}(\xi_1\bar{e}_5+\xi_2\bar{e}_6)=
\xi_1(f_0-\frac{a}{b}\log|\xi_1|)$ and
$\bar{p}=\bar{e}_0\bar{e}_4-\bar{e}_1\bar{e}_2$,
$\bar{\psi}(\xi_1\bar{e}_4+\xi_2\bar{e}_6)=
\xi_1(f_0-a\log|\xi_1|)$ respectively. The next step is to
consider quotients $\bar{\frakg}=\frakg/\fraka$ for two-%
dimensional ideals $\fraka\subset\langle e_4,e_5,e_6\rangle$
which brings along nothing new. Finally we consider
$\bar{\frakg}=\frakg/\langle e_4,e_5,e_6\rangle$ which is
primitive $\ast$-regular by Lemma~5.4 of~\cite{Ung1} because
$\bar{\frakn}=[\bar{\frakg},\bar{\frakg}]$ is commutative in
this case. Altogether we have shown that the $8$-dimensional
exponential Lie group~$G$ defined above is primitive
$\ast$-regular. In fact, we have thoroughly verified
assertions~(1) and~(2) set up in the beginning of this
section, which form our strategy for proving primitive
$\ast$-regularity of exponential Lie groups.

\begin{rem}
The preceding example indicates the prospects of success of
the approach involving separating triples $(W,p,\psi)$. It
shows the necessity to deal with sets of these triples and
to localize to an appropriate, sparse subset~$\Omega$
of~$\frakm^\ast$. This allows us to consider polynomial
functions $p$ which are $\Ad(M)$-invariant on~$\Omega$,
but not on the entire space.
\end{rem}

Finally we descend to our second example, the exponential
Lie algebra $\frakg=\langle d_0,d_1,e_0,\ldots,e_6\rangle$
with commutator relations $[e_1,e_2]=e_3$, $[e_1,e_3]=e_4$,
$[e_0,e_1]=-e_1$, $[e_0,e_2]=2e_2$, $[e_0,e_3]=e_3$,
$[d_0,e_0]=-ae_6$, $[d_0,e_2]=e_2$, $[d_0,e_3]=e_3$,
$[d_0,e_4]=e_4$, $[d_1,e_0]=-be_5$, and $[d_1,e_5]=e_5$
where $a,b\in\mR$. Its nilradical $\frakn=\langle e_1,%
\ldots,e_6\rangle$ is a trivial extension of the 3-step
nilpotent filiform algebra and
\[\frakg\underset{3}{\supset}\frakn\underset{2}{\supset}
C^1\frakn+\frakz\frakn\underset{2}{\supset}C^1\frakn
\underset{1}{\supset}C^2\frakn\underset{1}{\supset}
\{0\}\]
is a descending series of characteristic ideals of~$\frakg$.
The algebraic structure of~$\frakg$ is characterized by
the fact that the nilpotent subalgebra $\fraks=\langle%
d_0,d_1,e_0,e_6\rangle$ acts semi-simply on~$\frakn$ with
weights $\alpha$, $\gamma$, $\gamma-\alpha$,
$\gamma-2\alpha$, $\delta$, $0$ where $\alpha(e_0)=-1$,
$\gamma(d_0)=1$, $\delta(d_1)=1$, and the other values
are zero.\\\\
Let $\frakm=\langle e_0,\ldots,e_6\rangle$ and
$f\in\frakm^\ast$ be in general position such that
$\frakm=\frakm_f+\frakn$. In particular $f_\nu\neq 0$
for $4\le\nu\le 6$, even $f_\nu=1$ without loss of
generality. It follows from
\begin{align*}
\coAd(\exp(ve_1))f\;(e_3)&=f_3-vf_4\\
\coAd(\exp(xe_3))f\;(e_1)&=f_1+xf_4\\
\end{align*}
that we can establish $f_1=f_3=0$. Since
$\frakm=\frakm_f+\frakn$, there is some
$X=te_0+ve_1+we_2+xe_3+Z\in\frakm_f$ with $t\neq 0$.
Now $0=f([X,e_3])=tf_1+vf_4$ and $0=f([X,e_2])=tf_2+vf_3$
implies $v=0$ and $f_2=0$.\\\\
In coordinates $\Phi(t,w,x,Z)=\exp(te_0)\,\exp(ve_1)\,%
\exp(we_2+xe_3+Z)$ we compute
\begin{align*}
\coAd(\exp(rd_0)\,\exp(sd_1)\,\Phi(t,v,w,x,Z))f\;(e_0)
&=f_0+ar+bs-vx,\\
(e_1)&=e^t\,x,\\
(e_2)&=\frac{1}{2}e^{-(r+2t)}\,v^2,\\
(e_3)&=-e^{-(r+t)}\,v,\\
(e_4)&=e^{-r},\\
(e_5)&=e^{-s}
\end{align*}
which shows that the $\frakn^\ast$-closure $\Omega$ of
$X=\coAd(G)f$ is contained in the subset of all
$h\in\Omega$ such that $h_6=1$, $h_5\ge 0$, $h_4\ge0$,
and $p_3(h)=2h_2h_4-h_3h_3=0$. In particular $h\in\Omega$
and $h_4=0$ implies $h_3=0$.\\\\
Next we compute the relevant unitary representations of $M$.
Put $f_{r,s}=\coAd(\exp(rd_0)\exp(sd_1))f$ and
$\pi_{r,s}=\mcK(f_{r,s})$. Clearly $\frakp=\langle e_0,%
e_2,\ldots,e_6\rangle$ is a Pukanszky-Vergne polarization at
$f_{r,s}$ for all $r,s\in\mR$, and $\frakc=\langle e_1\rangle$
is a coexponential subalgebra for $\frakp$ in $\frakm$. Define
$\dot{e}_0=-ie_0+\frac{1}{2}e_6$ and $\dot{e}_\nu=-ie_\nu\in
\frakm_{\mC}$ for $1\le\nu\le 6$. It turns out that
$\pi_{r,s}=\ind_P^M\chi_{f_{r,s}}$ is infinitesimally
given by
\begin{align*}
d\pi_{r,s}(\dot{e}_0)&=(f_0+ar+bs)+\xi D_\xi,\\
d\pi_{r,s}(\dot{e}_1)&=-D_{\xi},\\
d\pi_{r,s}(\dot{e}_2)&=\frac{1}{2}e^{-r}\xi^2,\\
d\pi_{r,s}(\dot{e}_3)&=-e^{-r}\xi,\\
d\pi_{r,s}(\dot{e}_4)&=e^{-r},\\
d\pi_{r,s}(\dot{e}_5)&=e^{-s}.
\end{align*}
If $g\in\frakm^\ast$ such that $g_4=g_3=0$ and ($g_1\neq 0$
or $g_2\neq 0$), then $\frakn$ is a Pukanszky-Vergne
polarization at $g\in\frakm^\ast$ with coexponential
subalgebra $\frakc=\langle e_0\rangle$. It follows that
the infinitesimal operators of $\rho=\ind_N^M\chi_g$
are given by
\begin{align*}
d\rho(\dot{e}_0)&=\frac{i}{2}-D_\xi,\\
d\rho(\dot{e}_1)&=e^{\xi}\,g_1,\\
d\rho(\dot{e}_2)&=e^{-2\xi}\,g_2,\\
d\rho(\dot{e}_\nu)&=g_\nu\text{ for }3\le\nu\le 6.\\
\end{align*}
The explicit formulas for $\coAd(\Phi(t,v,x,Z))f$ suggest to
define the polynomial $p_1=e_0e_0e_4-2e_0e_1e_3+2e_1e_1e_2$
which is $\Ad(M)$-invariant on $\Omega\subset\frakm^\ast$,
but not on the entire space. This definition yields
$p_1(f_{r,s})=e^{-r}(f_0+ar+bs)^2$. We observe
that the equations for $d\pi_{r,s}(\dot{e}_\nu)$
bear a striking resemblance to the formulas for
$\coAd(\Phi(t,v,w,x,Z))f_{r,s}\,$: simply replace
$e^{-t}v$ by $\xi$ and $e^tx$ by $-D_\xi$. If we
define
\[W_1=\dot{e}_0\dot{e}_0\dot{e}_4-2\dot{e}_0\dot{e}_3
\dot{e}_1+2\dot{e}_2\dot{e}_1\dot{e}_1-i\dot{e}_3
\dot{e}_1\in\mcU(\frakm_{\mC}),\]
then it follows from $(\xi D_\xi)^2=\xi^2D_\xi^2-i\xi D_\xi$
and the binomial identity for the commuting operators
$\xi D_\xi$ and $f_0+ar+bs+\xi D_\xi$ that
$d\pi_{r,s}(W_1)=p_1(f_{r,s})\mdot\Id$. Moreover it is
easy to see that $d\pi(W_1)=p_1(h)\mdot\Id$ for all
$h\in\Omega$ and $\pi=\mcK(h)$, i.e., $(W_1,p_1)$ is a
Duflo pair \wrt~$\Omega$. The same is true for
$p_2=e_0e_4-e_1e_3$ and $W_2=\dot{e}_0\dot{e}_4
-\dot{e}_3\dot{e}_1$ with $d\pi_{r,s}(W_2)=p_2(f_{r,s})=%
e^{-r}(f_0+ar+bs)$. In addition we consider the functions
$\psi_1(\xi)=\xi_1(\,f_0-a\log|\xi_1|-b\log|\xi_2|\,)^2$
and $\psi_2(\xi)=\xi_1(\,f_0-a\log|\xi_1|-b\log|\xi_2|\,)$
on~$\frakz\frakm^\ast$ where $\xi\leftrightarrow%
\xi_1e^\ast_4+\xi_2e^\ast_5+\xi_3e^\ast_6$.
Now we can prove

\begin{prop}\label{cFm_prop:fili_characterize_closure_of_X}
Assume that \boldmath$b=0$\unboldmath. In this case $G$ is
primitive $\ast$-regular. If $f\in\frakm^\ast$ in general
position satisfies the stabilizer condition
$\frakm=\frakm_f+\frakn$, then $\{\,(W_\nu,p_\nu,\psi_\nu):
\nu=1,2\,\}$ is a set of separating triples for $X=\coAd(G)f$
in its $\frakn^\ast$-closure $\Omega$.
\end{prop}

\begin{proof}
Since $b=0$, the results of Section~\ref{sec:more_about_%
Fourier_multipliers} imply that $\psi_1$ and $\psi_2$ are
central Fourier multipliers. By definition $p_\nu(h)=
\psi_\nu(h\,|\,\frakz\frakm)$ for all $h\in X$, and hence
for all $h\in\overline{X}$. In order to prove the opposite
implication we suppose that $h\in\Omega$ such that
$p_\nu(h)=\psi_\nu(h\,|\,\frakz\frakm)$. Choose
sequences $r_n,\ldots,x_n$ such that $f'_n\to h'$ for
\[f_n=\coAd(\,\exp(r_nd_0)\exp(s_nd_1)%
\Phi(0,v_n,w_n,x_n,0)\,)f.\]
At first we assume that $h_4\neq 0$. Since $h\in\Omega$
and $p_2(h)=\psi_2(h\,|\,\frakz\frakm)$, it follows
$e^{-r_n}\to h_4\neq 0$ and
\begin{multline*}
e^{-r_n}f_{n0}=e^{-r_n}(f_0+ar_n-v_nx_n)
=\psi_2(f_n\,|\,\frakz\frakm)+f_{n1}f_{n3}\\
\to\psi_2(h\,|\,\frakz\frakm)+h_1h_3=h_0h_4
\end{multline*}
which implies $f_{n0}\to h_0$ and thus
$f_n\to h\in\overline{X}$. Next we assume $h_4=h_3=0$ so
that $r_n\to+\infty$. Note that $h\in\overline{X}$ implies
\[p_1(f_n)=e^{-r_n}(f_0+ar_n)\to 0=p_1(h)=2h_1^2h_2\]
which yields $h_1=0$ or $h_2=0$. Hence we must
distinguish three subcases. If $h_1\neq 0$ and
$h_2=0$, then we choose $r_n$, $s_n$ as above and
define $v_n=\frac{1}{h_1}(f_0+ar_n-h_0)$ and $x_n=h_1$.
By definition $f_n\to h\in\overline{X}$. The second
subcase is $h_1=0$ and $h_2\neq 0$. Since $h\in\Omega$,
it follows $h_2>0$. If we choose $r_n$, $s_n$ as above
and define $v_n=(2h_2)^{1/2}e^{r_n/2}$ and
$x_n=(2h_2)^{-1/2}e^{-r_n/2}(f_0-ar_n-h_0)$, then we
obtain $f_n\to h\in\overline{X}$. Finally we assume
$h_1=h_2=0$. Here we define $v_n=e^{r_n/4}$ and
$x_n=e^{-r_n/4}(f_0+ar_n-h_0)$. This proves
$f_n\to h\in\overline{X}$ in the third subcase. Altogether
we have shown that $h\in\overline{X}$ if and only if
$h\in\Omega$ and $p_\nu(h)=\psi_\nu(h\,|\,\frakz\frakm)$,
i.e., $\{(W_\nu,p_\nu,\psi_\nu):\nu=1,2\}$ is a set of
separating triples for~$X$ in~$\Omega$. From this and
the results of~\cite{Ung1} it follows that $\ker_{\CG}\pi$
is $L^1$-determined for all $\pi$ in general position. The
proper quotients of $G$ can be treated in analogy to the
proof of Lemma~\ref{cFm_lem:g52_characterize_closure_of_X}.
We omit the details because this would not yield anything
new.\qed
\end{proof}

The situation is more delicate if $b\neq 0$. In this
case the functions $\psi_1$ and  $\psi_2(\xi)=\xi_1%
(\,f_0-a\log|\xi_1|-b\log|\xi_2|\,)$ fail to be central
Fourier multipliers because of their singularity in~%
$\xi_2=0$. Thus we put $\tilde{\psi}_\nu(\xi)=%
\psi_\nu(\xi)\xi_2$, $\tilde{p}_\nu=p_\nu e_5$,
and $\tilde{W}_\nu=W_\nu e_5$. Note that
$(\tilde{W}_\nu,\tilde{p}_\nu)$ is a Duflo
pair \wrt~$\Omega$. Furthermore we define the
admissible part $\Omega_0=\{h\in\Omega:h_5\neq 0\}$
of $\Omega$. All we can prove is

\begin{lem}
Assume that \boldmath$b\neq 0$\unboldmath. If
$f\in\frakm^\ast$ is in general position such
that $\frakm=\frakm_f+\frakn$, then
$\{\,(\tilde{W}_\nu,\tilde{p}_\nu,
\tilde{\psi}_\nu):\nu=1,2\,\}$ is a set of
separating triples for the orbit $X=\coAd(G)f$
in the admissible part $\Omega_0$ of its
$\frakn^\ast$-closure $\Omega$. The non-admissible
part of the closure of $X$ is characterized as follows:
$h\in\overline{X}\setminus\Omega_0$ if and only if
$h\in\Omega$ and $h_5=h_4=h_3=0$.
\end{lem}

\begin{proof}
Clearly $\tilde{\psi}_1$ and $\tilde{\psi}_2$ are
central Fourier multipliers and $\tilde{p}_\nu(h)=
\tilde{\psi}_\nu(h\,|\,\frakz\frakm)$ for all
$h\in\overline{X}$. As in the proof of Proposition~%
\ref{cFm_prop:fili_characterize_closure_of_X} one can
show that $h\in\Omega_0$ and $\tilde{p}_\nu(h)=
\tilde{\psi}_\nu(h\,|\,\frakz\frakm)$ for
$\nu\in\{1,2\}$ implies $h\in\overline{X}$. Here
one heavily uses the fact that $h\in\Omega_0$ and
$f'_n\to h'$ implies the convergence of $s_n$ because
$e^{-s_n}\to h_5\neq 0$. Consequently the
$(\tilde{W}_\nu,\tilde{p}_\nu,\tilde{\psi}_\nu)$
are separating triples for $X$ in $\Omega_0$.\\\\
Next we verify the characterization of the non-admissible
part of the closure of $X$. Assume $h\in\overline{X}\setminus%
\Omega_0$. Then $h\in\Omega$, $h_5=0$, and $s_n\to+\infty$.
If $h_4$ were non-zero, then the sequences $r_n$, $v_n$,
$x_n$ would converge in contradiction to
$f_{n0}=f_0+ar_n+bs_n-v_nx_n\to h_0$. Thus $h_4=h_3=0$.\\\\
For the opposite implication we assume $h_5=h_4=h_3=0$.
If $h_1=h_2=0$, then we choose $r_n=s_n=n$, $v_n=e^{r_n/4}$,
and $x_n=e^{-r_n/4}(f_0+ar_n+bs_n-h_0)$. If $h_1\neq 0$
and $h_2=0$, then we put $r_n=s_n=n$,
$v_n=\frac{1}{h_1}(f_0+ar_n+bs_n-h_0)$, and $x_n=h_1$.
If $h_1=0$ and $h_2\neq 0$, then we define $r_n=s_n=n$,
$v_n=(2e^{r_n}h_2)^{1/2}$, and $x_n=(2e^{r_n}h_2)^{-1/2}%
(f_0+ar_n+bs_n-h_0)$. The last case is $h_1\neq 0$ and
$h_2\neq 0$. Here we choose $r_n=n$, $v_n=\sgn(bh_1)
(2e^{r_n}h_2)^{1/2}$, $x_n=h_1$, and
$s_n=\frac{1}{b}(h_0-f_0-ar_n+v_nx_n)$ so that
$s_n\to+\infty$. In any case it follows $f_n\to h\in
\overline{X}$.\qed
\end{proof}

At this stage it remains open whether $\bigcap_{r,s\in\mR}
\ker_{\LM}\pi_{r,s}\not\subset\ker_{\LM}\rho$ holds for
non-admissible, critical $g$ and $\rho=\mcK(g)$. Note that
in this particular case $g\in\Omega\setminus(\overline{X}%
\cup\Omega_0)$ if and only if $g\in\Omega$, $g_4\neq 0$,
and $g_5=0$. Although one might expect this 9-dimensional
exponential Lie group $G$ to be primitive $\ast$-regular
for $b\neq 0$, the results of the preceding sections are
too coarse to prove this. The preceding examples (and
similar ones) put the scope of the method of separating
triples into perspective.\\\\
What many exponential Lie algebras $\frakg$ of dimension
$\le 7$ have in common is that they contain ideals
$[\frakg,\frakg]\supset\frakb\underset{1}{\supset}
\fraka\underset{1}{\supset}\frakz\frakb$ where $\frakb$
is a 3-dimensional Heisenberg algebra, $\fraka$ is
commutative, and $\frakz\frakb$ is the one-dimensional
center of $\frakb$. In particular $\frakb\subset\frakn$
and $\frakz\frakb\subset\frakz\frakm$. Here we distinguish
the central case $\frakz\frakb\subset\frakz\frakg$ and
the non-central case $\frakz\frakb\not\subset\frakz\frakg$.
Roughly spoken, at least in low dimensions, the situation
is as follows:

\begin{rem}
Let $\frakg$ be an exponential Lie algebra, $\frakn$
a coabelian, nilpotent ideal of $\frakg$, and
$f\in\frakg^\ast$ in general position such that its
stabilizer $\frakm=\frakg_f+\frakn$ is not nilpotent.
In this situation it is advisable to look for Duflo
pairs~$(W_\nu,p_\nu)$ on~$M$. The existence
of~$(W_\nu,p_\nu)$ is (more or less) an intrinsic
property of~$M$. If $\frakg=\fraks\ltimes\frakn$ is a
semi-direct sum of a commutative subalgebra $\fraks$
and the ideal $\frakn$, then the existence
of~$(W_\nu,p_\nu)$ suffices to prove that $G$ is
primitive $\ast$-regular. If $\frakg=\mR d\ltimes\frakm$,
i.e., in case of a one-parameter subgroup $\Ad(\exp(rd))$
acting on the stabilizer $\frakm$, the (finer) method of
separating triples applies and yields the primitive
$\ast$-regularity of $G$. But this approach may fail
as soon as $\dim\frakg/\frakm\ge 2$.
\end{rem}

Using some of the results combined in this article,
the author proved in his thesis that all exponential
Lie algebras up to dimension seven are primtive
$\ast$-regular. This severe restriction on the
dimension of $\frakg$ implies that either $\frakg=%
\fraks\ltimes\frakn$ or $\dim\frakg/\frakm=1$. It
should be well noted that no counter-example
seems to be known so far.
\begin{ack}
The author would like to thank Prof.~Dr.~D.~Pogunkte for
his support and Prof.\ Dr.\ D.\ M\"uller for helpful
remarks.
\end{ack}

\ifthenelse{\boolean{appendix}}{\newpage
\section{The non-central case}
The postulates  enforce that $r$ is the only among the
first $l+1$ variables (the arguments of $Q$) which may
tend to infinity. Profiting by the existence of these
polynomials $p_0$ and $p_1$, we now obtain the following
description of (the admissible part of) the closure of
the orbit $\coAd(G)f$.
$$f_{r,s}=\coAd\left(\,E(r,s)\,\right)f\;.$$
The modification (multiplication by $C_\nu$) of the elements
$W_0$ and $W_1$ of $\mcU(\frakm_\mC)$ is absolutely
necessary in order to avoid singularities of $\psi$.
Such singularities make the application of Theorem X
impossible, as we have already noticed.
There is no doubt about the significance of these
assumptions for our treatise. Concerning the orbit
space of the coadjoint action, these assumptions are
indispensable for a concrete characterization of the
closure of the orbit $\coAd(G)f$ in $\frakm^\ast$.
From the representation theoretical point of view,
these postulates guarantee that $(W,p,\psi)$ separates
$\rho$ from $\{\pi_{r,s}:(r,s)\in\mR^{m+1}\}$. 
Though including semi-direct sums $\frakg=\fraks\ltimes\frakn$,
the condition $[\fraks_c,\frakt]=0$ does not reach far beyond
the case $\dim\frakg/\frakm=1$, i.e., a one-parameter group
$\Ad(\exp rd_0)$ acting on the stabilizer $\frakm$,
non-trivially on the central ideal $C^k\frakn$.
If $\dim\frakg/\frakm>1$ and $g\in\frakm^\ast$ is critical
for $\coAd(G)f$, but not admissible with respect to $f$
and $\Gamma$, then we must admit that the situation remains
somewhat mysterious.\\\\}{}

\bibliography{/home/oungerma/publications/literature}{}
\bibliographystyle{elsart-num-sort}

\end{document}